\documentclass[12pt, reqno, a4paper]{amsart}

\title {Multiplicity free $\U(2)$-actions and triangles}

\author[Goertsches]{Oliver Goertsches}
\address[Goertsches]{Fachbereich Mathematik und Informatik der Philipps-Universit\"at Marburg}
\email{goertsch@mathematik.uni-marburg.de}

\author[Van Steirteghem]{Bart Van Steirteghem}
\address[Van Steirteghem]{Department Mathematik, FAU Erlangen-N\"urnberg}
\email{bartvs@math.fau.de}

\author[Wardenski]{Nikolas Wardenski}
\address[Wardenski]{Fachbereich Mathematik und Informatik der Philipps-Universit\"at Marburg}
\email{wardensn@staff.uni-marburg.de}

\usepackage{fullpage}
\usepackage{amssymb, amsmath}
\usepackage{mathrsfs}
\usepackage{enumerate}
\usepackage[colorlinks,breaklinks, allcolors=blue]{hyperref}
\usepackage{longtable}
\usepackage[width=.80\textwidth]{caption}
\usepackage{cleveref}
\usepackage{tikz}

\swapnumbers
\newtheorem{theorem}{Theorem}[section]
\newtheorem{lemma}[theorem]{Lemma}
\newtheorem{proposition}[theorem]{Proposition}

\theoremstyle{definition}
\newtheorem{definition}[theorem]{Definition}
\newtheorem{remark}[theorem]{Remark}
\newtheorem{example}[theorem]{Example}

\numberwithin{equation}{section}

\newcommand{\C}{\mathbb C}
\newcommand{\Z}{\mathbb Z}
\newcommand{\N}{\mathbb N}
\newcommand{\Q}{\mathbb Q}
\newcommand{\R}{\mathbb R}

\newcommand{\PP}{\mathbb P}

\newcommand{\Pc}{\mathcal{P}}
\newcommand{\Qc}{\mathcal{Q}}
\newcommand{\Oc}{\mathcal{O}}
\newcommand{\Rp}{\R_{\ge 0}}

\newcommand{\GL}{\mathrm{GL}}
\newcommand{\PGL}{\mathrm{PGL}}
\newcommand{\SL}{\mathrm{SL}}

\newcommand{\SO}{\mathrm{SO}}

\newcommand{\SU}{\mathrm{SU}}
\newcommand{\U}{\mathrm{U}}

\newcommand{\fk}{\mathfrak{k}}

\newcommand{\ft}{\mathfrak{t}}
\newcommand{\fu}{\mathfrak{u}}

\newcommand{\sv}{\mathsf{v}}

\DeclareMathOperator{\rk}{rk}

\DeclareMathOperator{\Hom}{Hom}

\DeclareMathOperator{\conv}{conv}
\DeclareMathOperator{\cone}{cone}

\DeclareMathOperator{\Vect}{Vect}

\newcommand{\<}{\langle}
\renewcommand{\>}{\rangle}

\newcommand{\onto}{\twoheadrightarrow}
\newcommand{\into}{\hookrightarrow}

\newcommand{\wm}{\Gamma}

\newcommand{\dw}{\Lambda^+}
\newcommand{\wl}{\Lambda}

\newcommand{\wc}{\ft_+}

\newcommand{\TC}{T^{\C}}
\newcommand{\hfb}[3]{#1 \times_{#2}#3}
\newcommand{\mupv}{\mu_{\PP(V)}}
\newcommand{\ompv}{\om_{\PP(V)}}
\newcommand{\mop}{\Pc}

\newcommand{\eps}{\varepsilon}
\newcommand{\om}{\omega}

\newcommand{\inn}{\subset}
\newcommand{\loccit}{{\em loc.cit.}}

\begin{document}

\begin{abstract}
We classify the compact, connected multiplicity free Hamiltonian $\U(2)$-manifolds with trivial principal isotropy group whose momentum polytope is a triangle.  
\end{abstract}

\maketitle

\section{Introduction}
A fundamental invariant of a compact and connected Hamiltonian $K$-manifold $M$, where $K$ is a compact connected Lie group, is its momentum polytope $\mop(M)$. In \cite{knop-autoHam}, F.~Knop showed that if $M$ is \emph{multiplicity free} (cf.\ \cref{def:mfm} below) then $\mop(M)$ together with the principal isotropy group of the $K$-action uniquely determines $M$. This assertion had been conjectured by Th.~Delzant in the 1990s. Knop also gave  necessary and sufficient conditions for a polytope to be the momentum polytope of such a multiplicty free manifold $M$. These conditions involve a representation theoretic object, called  \emph{weight monoid}, associated to smooth affine spherical varieties, which constitute a certain class of complex algebraic varieties equipped with an action of a complex reductive group. 

In this paper, we apply Knop's classification result  
in the case where $K =\U(2)$ and determine the compact and connected multiplicity free Hamiltonian $\U(2)$-manifolds whose momentum polytope is a triangle and whose principal istotropy group is trivial.  The result is summarized in \cref{table_triangle_manifolds}. In contrast to Knop's work, which yields local descriptions of the multiplicty free manifold ``above'' open subsets of the momentum polytope, we have found explicit, global descriptions of the $\U(2)$-manifolds under consideration. Our hope is that they constitute useful ``experimental data'' to study the following natural question: Which geometric information about a multiplicity free manifold $M$ can ``directly'' be read off its momentum polytope $\mop(M)$?

In \cref{sec:background}, we review basic facts about Hamiltonian actions and recall, in \cref{thm:sjamaar_mp}, a special case of a local description, given by R.~Sjamaar in \cite{sjamaar-convexreex}, of the momentum polytope of a Hamiltonian manifold.  We also provide the necessary background to be able to state, in \cref{thm:knop_classif_mf}, a special case of Knop's aforementioned classification theorem which is adapted to our setting. \Cref{ex:u2} establishes notation we will use in \cref{sec:mfu2,sec:triangles,sec:difftypes}. 
The first purpose of \cref{sec:mfu2} is to further specialize Knop's \cref{thm:knop_classif_mf} to the case $K=\U(2)$: the classification of smooth affine spherical $(\SL(2,\C) \times \C^{\times})$-varieties from \cite{ppvs} yields \cref{prop:mom_polytopes}, which gives an elementary and explicit characterization of the momentum polytopes of compact and connected multiplictity free $\U(2)$-manifolds
with trivial principal isotropy group. A first application is \cref{prop:invcompl}, which extends the applicability of the  K\"ahlerizability criterion \cite[Theorem 8.8]{woodward-spherical}  due to C.~Woodward. We also extend \cite[Theorem 9.1]{woodward-spherical} and show in \cref{prop:atiyah}, using the extension criterion of S.~Tolman's \cite{tolman-exnonkaehlertorusactions}, that a multiplicity free $\U(2)$-manifold with trivial principal isotropy group carries a $\U(2)$-invariant compatible complex structure if and only if it carries a $T$-invariant compatible complex structure, where $T$ is a maximal torus of $\U(2)$.
We then apply \cref{prop:mom_polytopes} to find  in \cref{prop:triangles} the list of all triangles which occur as momentum polytopes of multiplicity free $\U(2)$-manifolds with trivial principal isotropy group. The rest of \cref{sec:triangles} is devoted to the proof of \cref{thm:triangles_realisation}: for each such triangle we explicitly and globally describe the corresponding compact and connected multiplicity free $\U(2)$-manifold. Finally, in \cref{thm:difftypes} of \cref{sec:difftypes}, we show  
that exactly four nonequivariant diffeomorphism types occur among these manifolds.  

We have tried to keep the exposition explicit and elementary in order to make our results and the employed techniques, which come from different areas of mathematics, accessible to as many readers as possible.  The techniques can directly be applied to the other compact Lie groups of rank $2$ and should yield analogous classifications and results.

\subsection*{Notation}
We use the convention that $0 \in \N$. From \cref{sec:mfu2} onward, $T$ will be the maximal torus of $\U(2)$ consisting of diagonal matrices and $\TC$ the subgroup of diagonal matrices in $\GL(2):=\GL(2,\C)$. We will use the notation from \cref{ex:u2} throughout the paper. 

Unless otherwise stated, $K$ will denote a compact connected Lie group and $G = K^{\C}$ its complexification, which is a complex connected reductive linear algebraic group.

\subsection*{Acknowledgment} The authors thank Guido Pezzini for helpful discussions, and in particular for suggesting the useful embedding in  \cref{prop:mfd_triangles_halfright_1}(\ref{prop:mfd_triangles_halfright_1_item1}). Van Steirteghem received support from the City University of New York PSC-CUNY Research Award Program and Wardenski from the German Academic Scholarship Foundation.

\section{Background}  \label{sec:background}
\subsection*{Multiplicity free manifolds} 
In this section we  review basic facts about Hamiltonian actions on symplectic manifolds, mostly following \cite{sjamaar-convexreex}, and then state results of Sjamaar (\cref{thm:sjamaar_mp}) and of Knop (\cref{thm:knop_classif_mf}) that will be essential in the proof of our main result. 

We begin with some basic notions and facts from the theory of compact Lie groups. Let $T$ be a maximal torus in the compact, connected Lie group  $K$. We will use $\fk$ and $\ft$ for the Lie algebras of $K$ and $T$, respectively. Furthermore $\fk^*$ and $\ft^*$ are the dual vector spaces, and we equip $\fk^*$ with the coadjoint action of $K$. We can and will view $\ft^*$ as a subspace of $\fk^*$ using the identification
\[\ft^* \cong (\fk^*)^T \inn \fk^*\]
with the subspace of $T$-fixed vectors in $\fk^*$.  We denote the weight lattice of $K$ by $\wl$, that is
\[\wl = \Hom_{\Z}(\ker(\exp|_\ft),\Z) \inn \ft^*,\] where $\exp: \fk \to K$ is the exponential map.  Note that 
\[\wl \to \Hom(T,\U(1)), \nu \mapsto [\exp(\xi) \mapsto \exp(2\pi\sqrt{-1} \<\nu,\xi\>)]\] is a bijection between $\wl$ and the character group of $T$, with inverse map
\[\Hom(T,\U(1)) \to \wl, \lambda \mapsto \frac{1}{2\pi \sqrt{-1}}\lambda_*,\]
where $\lambda_*$ is the derivative of $\lambda$ at the identity. We will use this bijection to identify $\wl$ with $ \Hom(T,\U(1))$.  In particular, if $V$ is a complex representation of $K$ and $v \in V$ is a weight vector of weight $\lambda \in \wl$, then we have (with abuse of notation)
\begin{align*}
\xi\cdot v &= 2\pi\sqrt{-1}\lambda(\xi) v \text{ for all }\xi \in \ft, \text{ and }\\
t\cdot v& = \lambda(t)v \text{ for all }t \in T.
\end{align*}

Next, we let $G:= K^{\C}$ be the complexification of $K$. Then $G$ is a complex connected reductive group of which $K$ is a maximal compact subgroup and of which the complexification $\TC$ of $T$ is a maximal torus.  
Recall that the weight lattice $\Hom(T^{\C},\C^{\times})$ of $G$ can be identified with $\wl$ using  the  restriction map 
\[\Hom_{\text{alg.gp.}}(T^{\C},\C^{\times}) \to \Hom_{\text{Lie gp.}}(T,U(1)), \lambda \mapsto \lambda|_T.\]
Fix a maximal unipotent subgroup $N$ of $G$ which is normalized by $\TC$  and let $\wc$ be the (closed) Weyl chamber in $\ft^*$ which is positive with respect to $N$. It is a fundamental domain for the coadjoint action of $K$ on $\fk^*$ and for the natural action of the Weyl group 
\[W:= N(T)/T\]
of $K$ on $\ft^*$.  Then 
\[\dw := \wl \cap \wc\]
is the monoid of dominant weights. Highest weight theory tells us that  the assignment
\[ V \mapsto \text{the weight of the $T$-action on $V^N$}\]
is a bijection between the set of isomorphism classes of  irreducible finite-dimensional complex representations of $K$ and $\dw$.  When $\lambda \in \dw$, we will write $V(\lambda)$ for the (up to isomorphism) unique  irreducible finite-dimensional complex representation of $K$ with highest weight $\lambda$. Furthermore, $K$ and $G$ have the same finite-dimensional complex representations: if $\dim_{\C} V
< \infty$ and $\rho: K \to \GL(V)$ is a homomorphism of Lie groups then there exists a unique homomorphism $\overline{\rho}: G \to \GL(V)$ of algebraic groups such that $\overline{\rho}|_K = \rho$.

\begin{example} \label{ex:u2}
To illustrate the objects we just recalled and to fix notation that we will use in what follows, we explicitly describe the objects in the case where $K$ is the unitary group  $\U(2)$ of rank $2$.  We choose the maximal torus 
\[T = \left\lbrace \begin{pmatrix}
t_1 & 0 \\ 0 & t_2
\end{pmatrix}: t_1,t_2 \in \C, |t_1| = |t_2| = 1\right\rbrace \inn \U(2). \]
The complexification of $\U(2)$ is $ \GL(2):= \GL(2,\C)$ and that of $T$ is 
\[\TC = \left\lbrace \begin{pmatrix}
t_1 & 0 \\ 0 & t_2
\end{pmatrix}: t_1,t_2 \in \C^{\times}\right\rbrace \inn \GL(2). \]
We will write $\eps_1, \eps_2$ for the basis of $\ft^*$ dual to the basis 
\[\xi_1:=\begin{pmatrix}
2\pi\sqrt{-1} & 0 \\ 0 & 0
\end{pmatrix},\ \xi_2:=\begin{pmatrix}
0 & 0 \\ 0 & 2\pi\sqrt{-1}
\end{pmatrix}\]
of $\ft$. Then the weight lattice is
\[\wl = \<\eps_1, \eps_2\>_{\Z}\]
and viewed as elements of $\Hom(T,\U(1))$ or of $\Hom(\TC,\C^{\times})$ the characters $\eps_1, \eps_2$ are defined by 
\begin{equation}
\eps_i  \begin{pmatrix}
t_1 & 0 \\ 0 & t_2
\end{pmatrix} = t_i \quad \text{ for } i \in \{1,2\}.
\end{equation}
For $N$ we choose the subgroup
\[\left\lbrace \begin{pmatrix}
1 & a \\ 0 & 1
\end{pmatrix} : a \in \C \right\rbrace
\]
of $\GL(2)$. The corresponding Weyl chamber is then
\[\ft_+ = \{\lambda \in \ft^* : \<\alpha^{\vee},\lambda\> \ge 0\},\]where 
\begin{equation} \label{eq:coroot}
\alpha^{\vee}:=\xi_1 - \xi_2
\end{equation}
is the coroot of the simple root 
\begin{equation} \label{eq:simpleroot}
\alpha:= \eps_1 - \eps_2 \in \wl \inn \ft^*
\end{equation}
of $\U(2)$ (and of $\GL(2)$). 

The Weyl group $W$ of $\U(2)$ (and of $\GL(2)$) is isomorphic to the symmetric group $S_2$ and the nontrivial element $s_{\alpha} \in W$ acts on $\ft^*$ by the reflection
\[s_{\alpha}(\lambda) = \lambda - \<\alpha^\vee,\lambda\>\alpha,\quad
\text{ where } \lambda \in \ft^*.\]

The monoid of dominant weights is 
\[\dw = \<\om_1, \om_2, -\om_2\>_{\N},\ \text{ where } \om_1 := \eps_1 \text{ and }\om_2:= \eps_1 + \eps_2.\]
Observe that $\om_1$ is the highest weight of the standard representation of $\U(2)$ (or of $\GL(2)$), which we will usually simply denote by $\C^2$. We will also use the notation $\C_{\det^k}$ for the one-dimensional representation $V(k\om_2)$, where $k\in \Z$: 
\[A \cdot z = \det(A)^kz\ \text{ for all } z \in \C_{\det^k}  \text{ and all }A \text{ in }\U(2) \text{ or in }\GL(2).\]
\end{example}

A \textbf{Hamiltonian $K$-manifold} is a triple $(M,\om,\mu)$, where $(M,\omega)$ is symplectic manifold equipped with a smooth $K$-action $K \times M \to M$ and a \textbf{momentum map} $\mu$, which means, by definition, a smooth map $\mu:M \to \fk^*$ that is $K$-equivariant with respect to the coadjoint action of $K$ on $\fk^*$ and satisfies
\[d\mu^\xi = \iota(\xi_M) \omega \text{ for all }\xi \in \fk.\]
Here $\xi_M$ is the vector field on $M$ defined by 
\[\xi_M(x) = \left.\frac{d}{dt}\right\vert_{t=0}\exp(t\xi)\cdot x \in T_x M \text{, where }x\in M,\] and $\mu^{\xi}:M \to \R$ is the function with $\mu^{\xi}(m) = \mu(m)(\xi)$.  
Since we have identified the Weyl chamber $\wc$ with a subset of $\fk^*$ we can define 
\begin{equation} \label{eq:def_mp}
\mop(M):= \mu(M) \cap \wc.
\end{equation}
In \cite[Theorem 2.1]{kirwan-convexity}, F.~Kirwan proved that $\mop(M)$ is the convex hull of finitely many points when $M$ is compact and connected. In that case we call $\mop(M)$ the \textbf{momentum polytope} of $M$.

\Cref{ex:mupv} describes an important source of Hamiltonian $K$-manifolds: projective spaces $\PP(V)$ associated to unitary representations $V$ of $K$.
\begin{example} \label{ex:mupv}
Let $V$ be a finite-dimensional unitary representation of $K$ with $K$-invariant Hermitian inner product $\<\cdot,\cdot\>$, where we adopt the convention that $\<\cdot,\cdot\>$ is complex-linear in the first entry. Following \cite[Ex.\ 2.1 and 2.2]{sjamaar-convexreex}, we describe well-known structures of Hamiltonian $K$-manifolds on $V$ and on the associated projective space $\PP(V)$, which is the space of complex lines in $V$. The map
\begin{equation} \label{eq:muv}
\mu_V: V \to \fk^*, \mu_V(v)(\xi) = \frac{\sqrt{-1}}{2} \<\xi v,v\>,
\end{equation}
where $\xi \in \fk$, is a momentum map for the $K$-invariant symplectic form $\om_V(\cdot,\cdot) = - \operatorname{Im}\<\cdot,\cdot\>$  on $V$. 
The Fubini-Study symplectic form $\ompv$ on $\PP(V)$ corresponding to $\<\cdot,\cdot\>$ is invariant under the natural $K$-action on $\PP(V)$ and we equip $\PP(V)$ with the momentum map
\begin{equation} \label{eq:mupv}
\mupv: \PP(V) \to \fk^* , \mupv([v])(\xi) = \frac{\sqrt{-1}}{2\pi} \frac{\<\xi v,v\>}{\Vert v \Vert^2},\end{equation}
where $\xi \in \fk$ and $[v]$ is the complex line through $v \in V \setminus\{0\}$. 

If $K$ is a torus and $v \in V$ is a weight vector with weight $\lambda$, then $\mu_V(v) = -\pi \|v\|^2v$ and $\mupv([v]) = -\lambda \in \fk^*$. This implies that
\begin{equation} \label{eq:weightscone}
\mu_V(V) = -\cone\{\lambda_1, \lambda_2, \ldots,\lambda_r\}
\end{equation}
and that 
 the momentum polytope of $(\PP(V), \mupv)$ is
\begin{equation}
\mop(\PP(V)) = \mupv(\PP(V)) = -\conv(\lambda_1,\lambda_2,\ldots,\lambda_r),
\end{equation}
where $\lambda_1,\lambda_2,\ldots,\lambda_r$ are the weights of $K$ in $V$. 
\end{example}

The following Theorem, which is due to Sjamaar and which was extracted from \cite{sjamaar-convexreex}  will be useful in \cref{sec:triangles}. In order to state it, we recall that the \emph{symplectic slice} of a Hamiltonian $K$-manifold $M$ in $m \in M$ is the symplectic vector space
\begin{equation}
N_m: = (T_m(K\cdot m))^{\perp}/(T_m(K\cdot m) \cap   (T_m(K\cdot m)^{\perp}),
\end{equation}
where  $(T_m (K\cdot m))^{\perp}$ is the symplectic annihilator of  $T_m (K\cdot m)$ in $T_m M$. The isotropy action of $K_m$ on $T_m M$ induces a natural symplectic representation of $K_m$ on $N_m$.

\begin{theorem}[{\cite{sjamaar-convexreex}}] \label{thm:sjamaar_mp}
Let $(M,\mu)$ be a compact connected Hamiltonian $K$-manifold. 
\begin{enumerate}[(a)]
\item If $m \in M$ such that $\mu(m)$ is a vertex of $\mop(M)$ lying in the interior of $\ft_+$, then $K_m = T$.  \label{item_pol_vertex_interior}%
\item Let $m \in M$ such that $\mu(m)$ lies in the interior of $\ft_+$ and $K_m = T$. Then 
\begin{equation} N_m=(T_m(K\cdot m))^{\perp}  \cong T_m M/T_m(K\cdot m)
\end{equation}
as $T$-modules, where  $N_m$ is the symplectic slice of $M$ in $m$. If $\Pi_m$ is the set of weights of the symplectic $T$-representation $N_m$,
then the cone with vertex $\mu(m)$ spanned by $\mop(M)$ is equal to $\mu(m) -\operatorname{cone}\Pi_m$. \label{item_pol_vertex_cone}
\end{enumerate}
\end{theorem}
\begin{proof}
Assertion (\ref{item_pol_vertex_interior}) is contained in part 2.\ of \cite[Theorem 6.7]{sjamaar-convexreex}.
Assertion (\ref{item_pol_vertex_cone})  follows from part 1. of \loccit\ and from \eqref{eq:weightscone} above; see also the paragraph in \cite{sjamaar-convexreex} containing Equation (6.9). To apply   \eqref{eq:weightscone} to the symplectic $T$-representation $N_m$  we recall that any symplectic $T$-representation $(V,\om_V)$  can be made into a unitary representation by choosing a $T$-invariant complex structure on $V$ that is compatible with the symplectic form $\omega_V$ and  that the weights of the representation are independent of this choice.
\end{proof}

\begin{remark} \label{rem_thm_sjam_mp}
In both parts of this remark, the point $m \in M$ is as in part (\ref{item_pol_vertex_cone}) of \cref{thm:sjamaar_mp}. 
\begin{enumerate}[(a)]
\item \label{rem_thm_sjam_mp_a}The cone with vertex $\mu(m)$ spanned by $\mop(M)$ is not pointed when $\mu(m)$ is not a vertex of $\mop(M)$ (we recall that a cone is called \emph{pointed} when it does not contain any line).
\item Later in this paper we will use that there exists a $K$-invariant diffeomorphism $\varphi$ from the homogeneous fiber bundle $\hfb{K}{T}{N_m}$ onto a $K$-invariant neighborhood of $K\cdot m$ in $M$ such that $\varphi([e,0]) = m$ (we recall the construction of  $\hfb{K}{T}{N_m}$ below in \cref{prop:hfb_diff}). This is an application of the slice theorem (see, e.g., \cite[Theorem 4.10]{kawakubo}). Actually, the proof of  \cref{thm:sjamaar_mp}(\ref{item_pol_vertex_cone}) uses the  \emph{symplectic} slice theorem of Marle~\cite{marle-construction} and Guillemin-Sternberg~\cite{guill&stern-sympphys} (see, e.g., \cite[Theorem 6.3]{sjamaar-convexreex} for a statement of this theorem).  \label{rem_thm_sjam_mp_slice}
\end{enumerate}
\end{remark}

We will also make use of the following well-known fact. For a proof, see, e.g., \cite[Theorems 1.2.1 and 1.2.2]{guill&sjam-crm}.
\begin{proposition} \label{prop:convexhullweylorbit}
Let $(M,\om,\mu)$ be a compact connected Hamiltonian $K$-manifold with momentum polytope $\mop(M)$ and let $r: \fk^* \to \ft^*$ be the dual map to the inclusion $\ft \to \fk$. Then $(M,\om, r \circ \mu)$ is a Hamiltonian $T$-manifold whose momentum polytope $\mop_T(M): = r (\mu (M))$ satisfies the equality 
\begin{equation}
\mop_T(M) = \conv \left(\bigcup_{w \in W} w \cdot \mop(M) \right).
\end{equation} 
\end{proposition}

\begin{definition} \label{def:mfm}
A \textbf{multiplicity free $K$-manifold} is a \emph{compact} and \emph{connected} Hamiltonian $K$-manifold $M$ such that 
\begin{equation} \label{eq:defmfm}
\mu^{-1}(a)/K_a \text{ is a point for every }a \in \mu(M).
\end{equation}
\end{definition} 

\begin{remark} \label{rem:mf}
\begin{enumerate}[(a)]
\item We have included connectedness and compactness in the definition of a multiplicity free $K$-manifold to avoid having to frequently repeat the associated adjectives in this paper. The (more general) notion of multiplicity free Hamiltonian manifold was introduced in \cite{mifo-liouville} and \cite{guill&stern-mf} as a Hamiltonian $K$-manifold  $M$  of which the Poisson algebra of $K$-invariant smooth functions $M \to \R$ is an abelian Lie algebra. Equivalent conditions on $M$ are given in \cite[Theorem 3]{huckleb&wurzb}. As shown in \cite[Proposition A.1]{woodward-classif}, for a compact, connected Hamiltonian $K$-manifold $M$ this original definition is
equivalent to condition \eqref{eq:defmfm}.  
\item \label{rem:mf_item_b} Let $(M,\om,\mu)$ be a compact connected Hamiltonian $K$-manifold. As observed in \cite{knop-autoHam}, just after Definition 2.1, $M$ is multiplicity free if and only if $$M/K \to \mop(M): K \cdot m \mapsto \mu(K\cdot m) \cap \wc$$ is a homeomorphism. Furthermore, if the principal isotropy group of the $K$-action on $M$ is discrete, then $M$ is multiplicity free if and only if 
\begin{equation}  \label{eq:dim_mf}
\dim(M) = \dim(K) + \rk(K),
\end{equation}
see \cite[Proposition A.1]{woodward-classif}.
\end{enumerate}
\end{remark}

In order to state  Knop's classification theorem for multiplicity free manifolds we introduce some additional notation and recall some more well-known facts. A smooth affine complex $G$-variety $X$ is called \textbf{spherical} if its ring of regular functions $\C[X]$ is multiplicity free as a $G$-module, that is 
\[\dim \Hom^G(V(\lambda),\C[X]) \leq 1 \text{ for all }\lambda \in \dw.\] The \textbf{weight monoid} $\wm(X)$ of $X$ is the set of highest weights of $\C[X]$, that is
\[\wm(X) :=\left \{\lambda \in \dw :  \Hom^G(V(\lambda),\C[X])\neq\{0\}\right\}. \] 
As proved by Losev in \cite[Theorem 1.3]{losev-knopconj}, a smooth affine spherical $G$-variety $X$ is uniquely determined by $\wm(X)$, up to $G$-equivariant isomorphism. 
If $a \in \ft_+ \inn \fk^*$ then the complexification $K_a^{\C}$ of the stabilizer $K_a$ of $a$ is a complex connected reductive subgroup of $G$. Since $K_a$ contains $T$ its weight lattice is still $\wl$. The Weyl chamber of $K_a$ and $K_a^{\C}$ corresponding to the maximal unipotent subgroup $N \cap K_a^{\C}$ of $K_a^{\C}$ is $\Rp(\ft_+ - a) \subset \ft^*$.  

\begin{example} \label{ex:u2kac}
We take $K=\U(2)$ and use the notation of \cref{ex:u2}. If $a \in \wc$ then
\[
K_a^{\C} = \begin{cases} \GL(2) & \text{if $\<\alpha^{\vee}, a\> = 0$;}\\
T^{\C} &\text{if $\<\alpha^{\vee},a\>>0$}
\end{cases}
\]
and the corresponding positive Weyl chamber of $K_a^{\C}$ is
\[
\Rp(\wc-a)= \begin{cases} \wc & \text{if $\<\alpha^{\vee}, a\> = 0$;}\\
\ft^* &\text{if $\<\alpha^{\vee},a\>>0$}
\end{cases}
\]
\end{example}

We can now specialize Knop's Theorems 10.2 and 11.2 from \cite{knop-autoHam} to the case of compact connected multiplicity free Hamiltonian manifolds with trivial principal isotropy group. 
\begin{theorem}[{Knop}] \label{thm:knop_classif_mf}
\begin{enumerate}[(a)]
\item \label{thm:knop_classif_mf_u_item} Suppose $(M,\om_M,\mu_M)$ and $(N,\om_N,\mu_N)$ are multiplicity free $K$-manifolds with trivial principal isotropy group. If $\mop(M) = \mop(N)$, then there exists a $K$-equivariant symplectomorphism $\varphi: M \to N$ such that $\mu_N \circ \varphi = \mu_M$. 
\item \label{thm:knop_classif_mf_e_item} Let $\Qc$ be a convex polytope in $\wc$. There exists a multiplicity free $K$-manifold $M$ with trivial principal isotropy group such that $\mop(M) = \Qc$ if and only if for every vertex $a$ of $\Qc$ there exists a smooth affine spherical $(K_a)^{\C}$-variety $X_a$ such that
\begin{align}
\label{eq:knop_classif_mf_tpi_item} &\wm(X_a)\text{ generates the weight lattice }\wl \text{ as a group, and}\\
\label{eq:knop_classif_mf_cone_item} &\Qc - a\text{ and }\wm(X_a)\text{ generate the same convex cone in }\ft^*.
\end{align}
\end{enumerate}
\end{theorem}

\begin{remark} \label{rem_after_knop_classif}
\begin{enumerate}[(a)]
\item The fact that the principal isotropy group of the $K$-action on $M$ is trivial is encoded in condition 
\eqref{eq:knop_classif_mf_tpi_item}
of \cref{thm:knop_classif_mf}. Knop's classification result \cite[Theorem 11.2]{knop-autoHam} makes no restrictions on the principal isotropy group, which is encoded as a sublattice of $\wl$.  
\item Part (\ref{thm:knop_classif_mf_u_item}) of \cref{thm:knop_classif_mf} is a special case of a conjecture due to Th.~Delzant. He proved his conjecture when $K$ is a torus in \cite{delzant-abel} and when $\rk(K) =2$ in \cite{delzant-rank2}. Knop proved it in general in \cite[Theorem 10.2]{knop-autoHam}.
\item \label{rem_after_knop_classif_item_pointed} Thanks to \cite{charwm}, the criterion in part (\ref{thm:knop_classif_mf_e_item}) of \cref{thm:knop_classif_mf} can be checked combinatorially (or algorithmically), i.e.\ without having to actually produce the spherical varieties $X_a$. On the other hand, in \cref{sec:mfu2} below we will distill from \cite{ppvs} all smooth affine spherical $\GL(2)$-varieties $X$ such that $\wm(X)$ generates $\wl$ as a group and the convex cone generated by $\wm(X)$ is pointed. 
\item \label{rem_after_knop_classif_item_locmod}Referring to \cite[Section 2]{knop-autoHam} for details, we briefly sketch how the $(K_a)^{\C}$-variety $X_a$ yields a ``local model'' of the manifold $M$  as in \cref{thm:knop_classif_mf}(\ref{thm:knop_classif_mf_e_item}). 
One can define a structure of Hamiltonian $K$-manifold on the homogeneous fiber bundle $\hfb{K}{K_a}{X_a}$ such that  a $K$-stable open subset of $\hfb{K}{K_a}{X_a}$ is isomorphic (as a Hamiltonian $K$-manifold) to a neighborhood of the $K$-orbit $\mu^{-1}(K\cdot a)$  in $M$.
\end{enumerate}
\end{remark}

\subsection*{Homogeneous fiber bundles} 
To explicitly describe multiplicity free $\U(2)$-manifolds in \cref{sec:triangles}, we will make use of \emph{homogeneous fiber bundles}, which are also known as \emph{associated bundles} or \emph{twisted products}. We recall their basic properties in the category of differentiable manifolds,  then in that of algebraic varieties, and finally state a comparison result that we will need later.

If $G$ is group, $H$ is a subgroup of $G$ and $F$ is a set on which $H$ acts, then we denote by $\hfb{G}{H}{F}$ the quotient set of $G\times F$ for the following action of $H$
\begin{equation} \label{eq:actionH}
h\cdot (g,f) = (gh^{-1},h\cdot f)\quad \text{for }g\in G,h\in H, f\in F.
\end{equation} As the left action of $G$ on $G\times F$, $g\cdot (g',f) = (gg',f)$ commutes with this action of $H$, we obtain a $G$-action on $\hfb{G}{H}{F}$. We will use $\pi$ for the $G$-equivariant quotient map
\[\pi: G\times F \to \hfb{G}{H}{F}, (g,f) \mapsto [g,f]\]
and $p$ for the (well-defined) $G$-equivariant map
\[p: \hfb{G}{H}{F} \to G/H, [g,f] \mapsto gH.\]

We begin with standard facts about the ``differentiable'' version of $\hfb{G}{H}{F}$ and sketch a proof for the sake of completeness. 
\begin{proposition} \label{prop:hfb_diff}
Let $G$ be a compact connected Lie group, $H$ a closed subgroup and $F$ a manifold equipped with a smooth action of $H$. Then the following hold:
\begin{enumerate}[(a)]
\item \label{prop:hfb_diff_item_structure} $\hfb{G}{H}{F}$ admits a unique structure as a manifold such that 
\begin{enumerate}[(i)]
\item $\pi: G\times F \to \hfb{G}{H}{F}$ is a smooth map; and
\item for an arbitrary manifold $N$ a map $h: \hfb{G}{H}{F} \to N$ is smooth if and only if $h \circ \pi$ is smooth. 
\end{enumerate} When $\hfb{G}{H}{F}$ is equipped with this structure, the map $p: \hfb{G}{H}{F} \to G/H$ and the action map $G \times (\hfb{G}{H}{F}) \to \hfb{G}{H}{F}$ are smooth;
\item \label{prop:hfb_diff_item_diffeo} If $f:M \to G/H$ is a smooth $G$-equivariant map,  where $M$ is a manifold equipped with a smooth action of $G$, and $A = f^{-1}(eH)$, then $A$ is a smooth $H$-invariant submanifold of $M$ and the map 
\[\hfb{G}{H}{A} \to  M, [g,a] \mapsto g \cdot a\]
is a $G$-equivariant diffeomorphism, if $\hfb{G}{H}{A}$ carries the manifold structure of part (\ref{prop:hfb_diff_item_structure}).
\end{enumerate}
\end{proposition}
\begin{proof}
The characterization of the manifold structure on $\hfb{G}{H}{F}$ in part (\ref{prop:hfb_diff_item_structure}) is a consequence of the basic fact that the quotient of a manifold under a free action of a compact Lie group is a manifold; see, for instance, \cite[Theorem 4.11]{kawakubo}. That $p$ is smooth now follows, because \(p \circ \pi: G \to G/H, g\mapsto gH\) is a smooth map, see e.g. \cite[Theorem 3.37]{kawakubo}. Furthermore, the aforementioned Theorem 4.11 in \cite{kawakubo} also tells us that $\pi$ has smooth local cross-sections, from which one can deduce that the action map is smooth. 
We turn to part (\ref{prop:hfb_diff_item_diffeo}).  One  observes that $eH$ is a regular value of $f$, that the two manifolds $\hfb{G}{H}{A}$ and $M$ have the same dimension and that the given map $\hfb{G}{H}{A} \to M$ is $G$-equivariant and injective. With standard arguments, one then shows that the map's differential is surjective everywhere.  
\end{proof}

Before stating an algebraic version of \cref{prop:hfb_diff} we recall that if $H$ is a closed algebraic subgroup of a linear algebraic group $G$, then the coset space $G/H$ carries a unique structure of algebraic variety such that the canonical surjection $G \to G/H$ is a so-called geometric quotient for the action of $H$ on $G$ from the right.  Equipped with this structure, as it always will be, $G/H$ is a smooth quasi-projective variety and the action map 
\[G \times (G/H) \to G/H, (g,g'H) \mapsto gg'H\]
is a morphism of algebraic varieties (see, for example, \cite[25.4.7 and 25.4.10]{tauvel-yu-laag}). 

The following proposition, which summarizes properties of the ``algebraic'' homogeneous fiber bundle, is extracted from \cite[Section 4.8]{popov&vinberg-invariant_theory}; see also \cite[Theorem 2.2]{timashev-embbook}.
\begin{proposition} \label{prop:hfb_algebraic}
Let $G$ be a complex connected reductive linear algebraic group, $H$ a closed algebraic subgroup and $F$ a smooth quasi-projective $H$-variety. Equip $\hfb{G}{H}{F}$ with the quotient Zariski-topology (i.e. the coarsest topology which makes $\pi: G\times F \to \hfb{G}{H}{F}$ continuous, where $G\times F$ carries the Zariski-topology) and with the  sheaf $\Oc$ which is the direct image under $\pi$ of the sheaf of $H$-invariant regular functions on $G \times F$. Then the following hold:
\begin{enumerate}[(a)]
\item The ringed space $(\hfb{G}{H}{F},\Oc)$ is a smooth complex algebraic variety;
\item The maps $\pi$ and $p$ and the action map $G \times (\hfb{G}{H}{F}) \to \hfb{G}{H}{F}$ are morphisms of algebraic varieties.
\end{enumerate}    
\end{proposition}

The next proposition recalls a standard fact in the theory of complex algebraic varieties, see \cite[Nr.\ 5]{GAGA}
\begin{proposition}\label{prop:analytic}
\begin{enumerate}[(a)]
\item If $X$ is a smooth complex algebraic variety, then $X$ admits a unique structure as a complex manifold such that every algebraic chart of $X$ is a holomorphic chart. We write $X^h$ for $X$ equipped with this structure of complex manifold. 
\item If $X$ and $Y$ are smooth complex algebraic varieties and $f: X \to Y$ a morphism of algebraic varieties, then $f: X^h \to Y^h$ is holomorphic.    
\end{enumerate}
\end{proposition}

A complex manifold  carries a natural structure of a differentiable manifold, by viewing the holomorphic charts as $\mathcal{C}^{\infty}$-charts. Thus \cref{prop:analytic} also equips every smooth algebraic variety with a structure of differentiable manifold, which we will call \textbf{standard}. Whenever we view a smooth algebraic variety as a differentiable manifold it will be equipped with this standard structure. In \cref{sec:triangles} we will make use of the following comparison result. We include a proof for completeness. 

\begin{proposition} \label{prop:hfb_comparison}
Consider the subgroup
\begin{equation} \label{eq:defB-}
B^-:=\left\lbrace\begin{pmatrix} a & 0 \\ c & d\end{pmatrix} : a,c,d \in \C, ad\neq 0\right\rbrace
\end{equation}
of $\GL(2)$ and recall the torus $T \inn \U(2)$ from \cref{ex:u2}. If $F$ is a smooth quasi-projective $B^{-}$-variety and we equip $\hfb{\GL(2)}{B^-}{F}$ with its standard structure as a differentiable manifold, then    
\begin{equation} \label{eq:compiso}
\U(2)\times_T F \to \GL(2)\times_{B-} F, [g,f] \mapsto [g,f]
\end{equation}
is a $\U(2)$-equivariant diffeomorphism. 
\end{proposition}

\begin{proof}
We equip $\GL(2)$ and $\GL(2)/B^-$ with their standard structures as differentiable manifolds.  Then  $\U(2)$ is a closed subgroup of the Lie group $\GL(2)$.  The inclusion $\U(2) \to \GL(2)$ induces a transitive smooth action of $\U(2)$ on $\GL(2)/B^-$.  Since $B ^- \cap \U(2) = T$ we thus obtain ---using \cite[Corollary 4.4]{kawakubo} for example--- a $\U(2)$-equivariant diffeomorphism $\U(2)/T \to \GL(2)/B^-$. Let $\varphi: \GL(2)/B^- \to \U(2)/T$ be the inverse diffeomorphism and recall the map
\[p: \hfb{\GL(2)}{B^-}{F} \to \GL(2)/B^-, [g,f] \to gB^{-}.\]
Then $\varphi \circ p : \hfb{\GL(2)}{B^-}{F} \to \U(2)/T$ is a $\U(2)$-equivariant smooth map. 
The claim now follows from (\ref{prop:hfb_diff_item_diffeo}) in \cref{prop:hfb_diff}.
\end{proof}

\begin{remark} The argument for \cref{prop:hfb_comparison} actually yields the following more general fact. Suppose $K$ is a compact connected Lie group and $G$ its complexification. Let $T$ be a maximal torus of $K$ and $B$ be a Borel subgroup of $G$ containing $T$. If $F$ is a quasi-projective $B$-variety and 
we equip $\hfb{G}{B}{F}$ with its standard structure as a differentiable manifold, then    
\[
\hfb{K}{T}{F} \to \hfb{G}{B}{F}, [k,f] \mapsto [k,f]
\]
is a $K$-equivariant diffeomorphism. 
\end{remark}

\section{Multiplicity free \texorpdfstring{$\U(2)$}{U(2)}-actions with trivial principal isotropy group}
\label{sec:mfu2}
Let $M$ be a multiplicity free $\U(2)$-manifold with trivial principal isotropy group. It follows from \cref{thm:knop_classif_mf} and \cref{ex:u2kac} that ``near'' a vertex lying on the wall of the Weyl chamber $\wc$, the momentum polytope $\mop(M)$ of $M$  ``looks like'' the weight monoid of a smooth affine spherical $\GL(2)$-variety. We distill \cref{table_gl2_pointed_wl} of all relevant smooth affine spherical $\GL(2)$-varieties from a result in  \cite{ppvs}. This then allows us to make the conditions \eqref{eq:knop_classif_mf_tpi_item} and \eqref{eq:knop_classif_mf_cone_item} very concrete in \cref{prop:mom_polytopes}.

\subsection*{Weight monoids of smooth affine spherical \texorpdfstring{$\GL(2)$}{GL(2)}-varieties} 
Table 5 in \cite{ppvs} contains all the smooth affine spherical $(\SL(2)\times \C^{\times})$-varieties and their weight monoids. We briefly explain how to use this classification to explicitly determine the weight monoids of smooth affine $\GL(2)$-spherical varieties. We will make use of the notation in \cref{ex:u2}. In particular, the weight lattice $\wl$ of $\GL(2)$ is spanned by the weights $\om_1,\om_2$.

As in \cite{ppvs} we choose 
\[H=\left\{\begin{pmatrix}
a & 0 \\ 0 & a^{-1}
\end{pmatrix}: a \in \C^{\times}\right\} \times \C^{\times}\] 
as the maximal torus and
\[U=\left\{\begin{pmatrix}
1 & b \\ 0 & 1
\end{pmatrix}: b \in \C\right\} \times \{1\}\] 
as the maximal unipotent subgroup of $\SL(2) \times \C^{\times}$ normalized by $H$. 
The weights $\om: H \to \C^{\times}$ and $\eps: H \to \C^{\times}$ defined by
\[\om\left(\begin{pmatrix}
a & 0 \\ 0 & a^{-1}
\end{pmatrix},z\right) = a, \text{ and } \varepsilon\left(\begin{pmatrix}
a & 0 \\ 0 & a^{-1}
\end{pmatrix},z\right) = z\]
span the weight lattice $\Hom(H,\C^{\times})$ of $\SL(2)\times \C^{\times}$ and the monoid of dominant weights corresponding to $U$ is
\[\<\om,\eps,-\eps\>_{\N} \subset \Hom(H,\C^{\times}).\] 
We define the isogeny
\begin{equation} \label{eq:isogeny}
\varphi\colon \SL(2) \times \C^{\times} \to \GL(2) \colon (A,z) \mapsto zA
\end{equation}
and denote the induced (injective) map $\wl \to \Hom(H,\C^{\times})$ on weight lattices by $\varphi^*$.  
Then $\wm \subset \dw$ is the weight monoid of a smooth affine spherical $\GL(2)$-variety if and only if $\varphi^*(\wm) \inn \<\om,\eps,-\eps\>_{\N}$ is the weight monoid of a smooth affine spherical $(\SL(2)\times \C^{\times})$-variety. 
Since 
\[\varphi^*(\om_1) = \om + \varepsilon \text{ and } \varphi^*(\om_2) = 2 \varepsilon\]
we have
\[
\varphi^*(\wl) = \{a\omega + b\varepsilon \colon a \equiv b \mod 2\}
\] and it follows that 
the images under $\varphi^*$ of the weight monoids of smooth affine spherical $\GL(2)$-varieties are exactly those weight monoids in  \cite[Table 5]{ppvs} which are subsets of
$ \{a\omega + b\varepsilon \colon a \equiv b \mod 2\}$.

In view of part (\ref{thm:knop_classif_mf_e_item}) of \cref{thm:knop_classif_mf}  we restrict ourselves to those weight monoids $\wm$ of smooth affine spherical $\GL(2)$-varieties such that the cone $\Rp\wm$ generated by $\wm$ is pointed (as defined in \cref{rem_thm_sjam_mp}(\ref{rem_thm_sjam_mp_a})) and such that $\Z\wm  = \wl$. This yields the weight monoids listed in \cref{table_gl2_pointed_wl}. 
In fact, in view of Knop's condition \eqref{eq:knop_classif_mf_cone_item}, we list the \emph{weight cones} $\Rp \wm \subset \wc$ instead of the weight monoids. The cone $\Rp\wm$ determines $\wm$ because we have fixed the lattice $\Z\wm$ generated by $\wm$ to be $\wl$ and because of the equality $\wm = \Z\wm \cap \Rp\wm$, which follows from the fact that smooth varieties are normal. In summary, these computations yield the following proposition.
\begin{proposition} \label{prop:GL2pointedcone}
If $X$ is a smooth affine spherical $\GL(2)$-variety such that $\Z\wm(X) = \wl$ and such that $\Rp \wm(X)$ is pointed, then $\Rp\wm(X)$ is one of the cones listed in \cref{table_gl2_pointed_wl}. 
\end{proposition}

\begin{remark}
As they provide local models of multiplicity free $\U(2)$-manifolds with trivial principal isotropy group, we have included in \cref{table_gl2_pointed_wl} the (unique) smooth affine spherical $\GL(2)$-varieties $X$ that realize the listed weight cones $\Rp\wm(X)$. We leave the verification that each variety $X$ in the table has the given weight cone to the reader. This can be deduced from \cite[Table 5]{ppvs} using the isogeny $\varphi$ defined in \eqref{eq:isogeny} or by using basic facts in the representation theory of $\GL(2)$ to determine the highest weights of $\GL(2)$ that occur in the coordinate ring $\C[X]$ of $X$. 
\end{remark}

\begin{longtable}{|c|c|c|c|}
\caption{Pointed weight cones of smooth affine spherical $\GL(2)$--varieties $X$ with  $\Z\wm(X)=\wl$. The ``Case'' numbers refer to those in \cite[Table 5]{ppvs}}\label{table_gl2_pointed_wl}\\
\hline
Case & $X$  & $\Rp\wm(X)$ &  parameters\\
\hline
11 &$\bigl(\C^2 \otimes \C_{\det^{-(k+1)}}\bigr) \times \C_{\det^{-\ell}}$ & $\cone((k+1)\varepsilon_1+k\varepsilon_2, \ell(\varepsilon_1+\varepsilon_2))$ & \begin{tabular}{l} $k \in \Z,$\\ $\ell \in \{1,-1\}$\end{tabular} \\
\hline
14 & $\GL(2) \times _{\TC} \C_{-(j\alpha+ \varepsilon_1)}$ &$\cone(\alpha, j\alpha + \varepsilon_1)$  & $j \in \N$ \\
\hline
14 & $\GL(2) \times _{\TC} \C_{-(j\alpha- \varepsilon_2)}$ &$\cone(\alpha, j\alpha - \varepsilon_2)$  & $j \in \N$ \\
\hline
15 &$\GL(2)\left/\left\lbrace \begin{pmatrix} z^j & 0 \\ 0 & z^{j+1}\end{pmatrix}: z \in \C^{\times}\right\rbrace \right.$&$\cone(j\alpha+\varepsilon_1, j\alpha-\varepsilon_2)$   &$j \in \N$ \\
\hline
\multicolumn{4}{|l|}{$\alpha = \varepsilon_1-\varepsilon_2$ as in \cref{ex:u2}.}
\\
\multicolumn{4}{|l|}{In Case 11, $\C^2$ stands for the defining representation of $\GL(2)$.}\\
\hline
\end{longtable}

\subsection*{Momentum polytopes}
In \cref{prop:mom_polytopes} we specialize Knop's \cref{thm:knop_classif_mf} to the case $K=\U(2)$. We continue to use the notation in \cref{ex:u2}. In particular, $\alpha = \varepsilon_1 - \varepsilon_2$ is the simple root of $\U(2)$. 
Combining \cref{table_gl2_pointed_wl} with \cref{thm:knop_classif_mf} we obtain the following.
\begin{proposition} \label{prop:mom_polytopes}
Let $\Pc$ be a convex polytope in $\ft_+$. Then $\Pc$ is the momentum polytope of a (unique) multiplicity free $\U(2)$-manifold with trivial principal isotropy group if and only if all of the following conditions are satisfied:
\begin{enumerate}[(1)]
\item $\dim \Pc = 2$; \label{item:mom_poly_1}
\item $\Pc$ is rational with respect to $\wl$, i.e. for every two vertices $a,b$ of  $\Pc$ connected by the edge $[a,b]$ of $\Pc$, the intersection  $\R(b-a) \cap \wl$ is nonempty (we will denote the primitive elements of $\wl$ on the extremal rays of the cone $\R_{\ge 0}(\Pc-a)$ by  $\rho_1^a, \rho_2^a$); \label{item:mom_poly_2}
\item (Delzant) If $a$ is a vertex of $\Pc$ with $\<\alpha^{\vee},a\> >0$, then $(\rho_1^a,\rho_2^a)$ is a basis of $\wl$; \label{item:mom_poly_3}
\item If $a$ is a vertex of $\Pc$ with $\<\alpha^{\vee},a\> =0$, then $\{\rho^a_1, \rho^a_2\}$ is one of the following sets: \label{item:mom_poly_4}
\begin{align}
& \{\varepsilon_1 + \varepsilon_2, k(\varepsilon_1 + \varepsilon_2) + \varepsilon_1\} \quad \text{ for some }k \in \Z; \label{eq:cone_1}\\
& \{-(\varepsilon_1 + \varepsilon_2), k(\varepsilon_1 + \varepsilon_2) + \varepsilon_1\} \quad \text{ for some }k \in \Z; \label{eq:cone_2}\\
& \{\alpha, j\alpha+\varepsilon_1\}  \quad \text{ for some }j \in \N; \label{eq:alpha_cone_1} \\
& \{\alpha, j\alpha-\varepsilon_2\}  \quad \text{ for some }j \in \N; \label{eq:alpha_cone_2}\\
& \{j\alpha+\varepsilon_1, j\alpha-\varepsilon_2\}  \quad \text{ for some }j \in \N. \label{eq:alpha_cone_3}
\end{align}
\end{enumerate} 
\end{proposition}

\begin{proof}
Thanks to Knop's \cref{thm:knop_classif_mf}, the cones in the third column of \cref{table_gl2_pointed_wl} describe the "local" shape, near a vertex that lies on the wall of the Weyl chamber $\wc$, of the momentum polytope $\mop$ of a multiplicity free $\U(2)$-manifold with trivial principal isotropy group. If $a$ is a vertex of $\mop$ that lies in the interior of $\wc$, we have $(K_a)^{\C} = T^{\C}$. The shape of $\mop$ near $a$ must be as  described in part (\ref{item:mom_poly_3}) of the proposition due to the well-known structure of the weight monoids of smooth affine toric varieties (see, e.g. \cite[Section 2.1]{fulton-toric}). The proposition follows.   
\end{proof}

In \cref{rem:fixpoints} we give some geometric information related to vertices of the momentum polytopes under consideration in \cref{prop:mom_polytopes}. We first introduce some additional notation. Suppose $(M,\mu)$ is a multiplicity free $\U(2)$-manifold with momentum polytope $\Pc$. Let 
\begin{equation}
\Psi:M \to \mop \inn \wc,\ m \mapsto  \mu(K\cdot m)\cap \wc \label{eq:eqvmm}
\end{equation}
be the \textbf{invariant momentum map} of $M$ and let $\mu_T:M \to \ft^*$ be the momentum map of $M$ considered as a $T$-manifold, that is $\mu_T = r \circ \mu$, where $r: \fk^* \onto \ft^*$ is the restriction map. We recall from \cref{rem:mf}(\ref{rem:mf_item_b}) that every fiber of $\Psi$ is a $K$-orbit and from  \cref{prop:convexhullweylorbit} that $\mu_T(M)$ is the convex hull of $\Pc \cup s_{\alpha}(\Pc)$. 

\begin{remark}[Vertices and fixpoints] \label{rem:fixpoints}
Let $\Pc$ be a polytope satisfying the conditions in \cref{prop:mom_polytopes} and let $M$ be the multiplicity free $\U(2)$-manifold with trivial principal isotropy group such that $\Pc(M)=\Psi(M)=\Pc$. The local models $X$ of $M$ given in \cref{table_gl2_pointed_wl} yield the following information (see also \cref{rem_after_knop_classif}(\ref{rem_after_knop_classif_item_locmod})). 
\begin{enumerate}[(a)]
\item If $a$ is as in case (\ref{item:mom_poly_3}) of \cref{prop:mom_polytopes}, $\Psi^{-1}(a)$ contains exactly two $T$-fixpoints $p_1,p_2$ and $\mu_T(p_1) = s_{\alpha}(\mu_T(p_2))$.
\item If the extremal rays at $a$ are those  in \cref{eq:cone_1} or \cref{eq:cone_2}, then there is a unique $T$-fixpoint $p$ in $\Psi^{-1}(a)$. Moreover $\mu_T(p) = a$ and $p$ is even fixed by $\U(2)$.
\item \label{fixponts_halfrefl} In the cases of  \cref{eq:alpha_cone_1,eq:alpha_cone_2}, there are exactly two $T$-fixpoints $p_1,p_2$ in $\Psi^{-1}(a)$. Moreover $\mu_T(p_1)=\mu_T(p_2) = a$. 
\item \label{fixpoints_refl} In the case of \cref{eq:alpha_cone_3}, $\Psi^{-1}(a)$ does not contain any $T$-fixpoints. 
\end{enumerate}  
\end{remark}

\subsection*{Invariant compatible complex structures} 
We recall that a complex structure $J$ on a symplectic manifold $(M,\om)$ is called \emph{compatible} if $(M,\om,J)$ is K\"ahler. In \cite{woodward-kaehler}, Woodward showed that Delzant's result in \cite[\S 5]{delzant-abel}, that all compact multiplicity free torus actions admit an invariant compatible complex structure, does not generalize to the non-abelian case (see also \cref{ex:woodward} below). 
In this section we present a generalization of a criterion of Woodward's for the existence of a $\U(2)$-invariant compatible complex structure on certain multiplicity free $\U(2)$-manifolds. More precisely, in \cite[Theorem 8.8]{woodward-spherical}, Woodward provided such a criterion for a  class of multiplicity free $\SO(5)$-manifolds and remarked that it could be adapted to other rank $2$ groups. In case the acting group is $\U(2)$, Woodward's criterion applies to those multiplicity free $\U(2)$-manifolds with trivial principal isotropy whose momentum polytope has a vertex on the wall of the Weyl chamber such that, near this vertex, the momentum polytope looks like one of the cones spanned by the vectors in \eqref{eq:alpha_cone_3}. Thanks to the work \cite{martens-thaddeus-nonabsc} of J.~Martens and M.~Thaddeus on non-abelian symplectic cutting
we can show that his criterion can be used to decide the existence of a $\U(2)$-invariant compatible complex structure for any multiplicity free $\U(2)$-manifold with trivial principal isotropy, see \cref{prop:invcompl}. 
In the proof of  \cref{prop:invcompl}  we use  the so-called \emph{extension criterion} of Tolman \cite{tolman-exnonkaehlertorusactions} to show in \cref{prop:atiyah} that a multiplicity free $\U(2)$-manifold with trivial principal isotropy group carries a $\U(2)$-invariant compatible complex structure if and only if it carries a $T$-invariant compatible complex structure. Woodward had proved an analogous statement for certain multiplicity free $\SO(5)$-manifolds in \cite[Theorem 9.1]{woodward-kaehler}. \Cref{prop:atiyah} also gives a second K\"ahlerizability criterion for our $\U(2)$-manifolds in terms of the $T$-momentum polytope and the images of the $T$-fixpoints under the $T$-momentum map.

Recall that the wall of the Weyl chamber $\wc$ of $\U(2)$ is its subset 
$\{\lambda \in \ft^*: \<\alpha^{\vee}, \lambda\>=0\}$, where $\alpha^{\vee}$ is the simple coroot as in \eqref{eq:coroot} of \cref{ex:u2}.

\begin{definition} \label{def:positive_face}
Let $\mop$ be a  $2$-dimensional polytope in the Weyl chamber $\ft_+$ of $\U(2)$, let $F$ be an edge of $\mop$ and let $\mathsf{n}_F$ be an inward-pointing normal vector to $F$. We call $F$ a \textbf{positive} edge of $\mop$ if $\<\alpha^{\vee}, \mathsf{n}_F\> > 0$.
\end{definition}

\begin{remark}
It follows from \cref{prop:mom_polytopes} that if the momentum polytope $\mop(M)$ of a multiplicity free $\U(2)$-manifold with trivial principal isotropy group has exactly one vertex $a$ on the wall of $\wc$, then $\Rp(\mop(M)-a)$ is the cone spanned by one of the sets $\{\rho_1^a, \rho_2^a\}$ in \cref{eq:alpha_cone_1,eq:alpha_cone_2,eq:alpha_cone_3} of that proposition. In particular, $\mop(M)$ has one or two positive edges that contain $a$. 
\end{remark}

Here is the announced generalization for $\U(2)$ of Woodward's K\"ahlerizability criterion \cite[Theorem 8.8]{woodward-spherical}. Its formal proof will be given on page ~\pageref{proof_prop_invcompl}. 
\begin{proposition} \label{prop:invcompl}
Suppose $M$ is a multiplicity free $\U(2)$-manifold with trivial principal isotropy group. Then $M$ admits a $\U(2)$-invariant compatible complex structure if and only if the following property holds: if the momentum polytope $\mop(M)$ of $M$ has exactly one vertex on the wall of $\wc$, then every positive edge of $\mop(M)$ contains that vertex.  
\end{proposition}

\begin{remark}
With a straightforward adaptation of the proof of Proposition 7.27 and Corollary 7.28 in \cite{mp} \cref{prop:invcompl} can also be proved by applying Theorem 7.16 of \loccit, which gives a (rather technical) general criterion for the existence of an invariant compatible complex structure on a multiplicity free manifold, using the combinatorial theory of spherical varieties. 
\end{remark}

\begin{example}[Woodward] \label{ex:woodward}
In \cite{woodward-kaehler}, Woodward showed that the multiplicity free $\U(2)$-manifold $M$ with momentum polytope 
\[\mop(M) = \conv(0, \eps_1, -\eps_2, 3\eps_1 - \eps_2)\]
is not K\"ahlerizable. This fact can be deduced immediately from \cref{prop:invcompl}: the edge of $\mop(M)$ connecting the vertices $\eps_1$ and $3\eps_1 - \eps_2$ is positive, but does not contain the vertex $0$ of $\mop(M)$ that lies on the wall of the Weyl chamber. 
A picture of $\mop(M)$ can be found on page~\pageref{Woodward}: it is the trapezoid with vertices $0, v_1, v_2$ and $v_3$ on the right in \cref{Woodward}.
Similarly, the multiplicity free $\U(2)$-manifold with momentum polytope
\[\conv(0, \varepsilon_1, \alpha, 3\varepsilon_1 - \varepsilon_2)\]
is not K\"ahlerizable, because the edge connecting $\varepsilon_1$ and $3\varepsilon_1 - \varepsilon_2$ is positive and does not contain the vertex $0$. This polytope is the trapezoid with vertices $v_0, v_1, v_2$ and $v_3$ on the right in \cref{fig:morexrays}. This kind of polytope was not covered by the criterion in \cite{woodward-kaehler}.
\end{example}

The following lemma establishes a first part of \cref{prop:invcompl}.
\begin{lemma} \label{lem:nosvonwallkaehler}
If $M$ is a multiplicity free $\U(2)$-manifold whose momentum polytope $\mop$ does not have exactly one vertex on the wall of the Weyl chamber $\ft_+$, then $M$ admits a $\U(2)$-invariant complex structure. 
\end{lemma}
\begin{proof}
Our strategy is inspired by \cite[\S 3]{woodward-kaehler}  and uses E.~Lerman's \emph{symplectic cutting} \cite{lerman-sc}; see also \cite{Ler_et_al_nonabconvex}. We start with a certain multiplicity free (non-compact) $\U(2)$-manifold $M_1$ admitting a second Hamiltonian action of the maximal torus $T$ of $\U(2)$ that commutes with the $\U(2)$-action and such that $\phi(m) = \Psi(m)$ for all $m \in M$, where $\phi: M_1 \to \ft^*$ is the momentum map of the second $T$-action. We then perform a sequence of symplectic cuts
(respecting the actions of both $\U(2)$ and $T$) until the momentum polytope has the desired shape $\mop$. Because $\mop$ is of Delzant type (by \cref{prop:mom_polytopes}), it can be obtained from $\phi(M_1) = \Psi(M_1)$ by a finite sequence of cuts along hyperplanes such that, at each stage, the corresponding symplectic cut yields a smooth manifold.  

We first assume that an entire edge of the momentum polytope $\mop$
lies on the Weyl wall. Let $\sv$ be one of the two vertices of $\mop$ on the Weyl wall and suppose that 
\[\cone(\mop - \sv) = \cone(-(\eps_1 +\eps_2), \eps_1+k(\eps_1+\eps_2)) \ \text{for some } k\in \Z.\] 
We set $M_1=\C^3$ and equip it with the $\U(2)$-action given by
\[
g\cdot ((z_1,z_2),z_3):=((\det(g)^{-(k+1)} \cdot g)\cdot (z_1,z_2), \det(g)\cdot z_3)
\]
(this is precisely the action in the first row of \cref{table_gl2_pointed_wl} for $\ell=-1$)
and the standard Hamiltonian $\U(2)$-structure as representation of $\U(2)$; see e.g.\ \cite[Example 2.1]{sjamaar-convexreex}.
The explicit expression for the invariant momentum map $\Psi$ is then
\begin{align*}
 \Psi:M_1 \rightarrow \ft_+, \quad \Psi(z_1,z_2,z_3)=\frac{\pi}{2} (|z_1|^2+|z_2|^2) 
((k+1)\eps_1+k\eps_2)- \frac{\pi}{2} |z_3|^2 (\eps_1+\eps_2).
\end{align*}
Indeed, the restriction of the momentum map of $M_1$ to the cross-section $0 \oplus \C\oplus \C$  takes values in $\ft^*$ and is thus given by the momentum map
of the $T$-action on $0 \oplus \C \oplus \C$, which is
\begin{align*}
	(0,z_2,z_3)\mapsto \frac{\pi}{2} |z_2|^2 ((k+1)\eps_1+k\eps_2)- \frac{\pi}{2} |z_3|^2 (\eps_1+\eps_2).
\end{align*}
As $\Psi$ is constant on $\U(2)$-orbits, it now follows that it is
of the asserted shape.

We also equip $M_1$ with the following additional action of $T$:
\begin{align*}
(t_1,t_2)\cdot (z_1,z_2,z_3)=(t_2 (t_1t_2)^{-(k+1)} z_1,t_2 (t_1t_2)^{-(k+1)} z_2,t_1 t_2 z_3). 
\end{align*}
This action is  effective and commutes with the $\U(2)$-action. More importantly,
it is Hamiltonian with momentum map $\phi$ equal to
$\Psi$. We can use $\phi$ to perform the aforementioned sequence of symplectic 
cuts until the momentum image of $\phi$ is equal to $\Pc$. Since $\phi$
and $\Psi$ coincide on $M_1$, they will coincide after every cut. 
By the uniqueness part of Knop's \cref{thm:knop_classif_mf}, the $U(2)$-manifold obtained at the end of this process is $M$.
Due to basic properties of the symplectic cut,
the manifold is still K\"ahler.

If $\mop$ lies in the interior of the Weyl chamber, then we can still start with (for example) $M_1$
(for some suitable choice of the parameter $k$) and we can again cut $\Psi(M_1) = \phi(M_1)$ until we reach $\mop$.
\end{proof}
\begin{remark} \label{rem:resultkaehler}
	The reason that the resulting manifold after all the symplectic cuts in the proof of \cref{lem:nosvonwallkaehler} is K\"ahler, is that the cuts are made with respect to circle subgroups of $T$ and that action of $T$ on $M$ with momentum map $\phi$ is a \emph{global} action which extends to an action of the complexification $T^{\C}$ of $T$.
	In Woodward's non-K\"ahlerizable example \cite[\S 3]{woodward-kaehler}, only a local circle action was used to perform the cut; in general
   one does not obtain a compatible K\"ahler structure after such a cut.
\end{remark}
\begin{remark}
The K\"ahler structures constructed in the proof of \cref{lem:nosvonwallkaehler} are
in fact toric K\"ahler structures for a torus of rank $3$. Indeed, the constructed manifold carries an induced $T\times T$-action after every cut. Its kernel  
always has dimension $1$, which means that it descends to a multiplicity free 
action of a torus of rank $3$.
\end{remark}

We now use \cite[Corollary 4]{martens-thaddeus-nonabsc} to establish the next part of \cref{prop:invcompl}.
\begin{lemma} \label{lem:not_too_many_pos_invcompl}
Let $M$ be a multiplicity free $\U(2)$-manifold with trivial principal isotropy group. Suppose that the momentum polytope $\mop$ of $M$ has exactly one vertex $a$ on the wall of $\wc$. If every positive edge of $\mop$ contains the vertex $a$, then $M$ admits a $\U(2)$-invariant compatible complex structure. 
\end{lemma}
\begin{proof}
It follows from \cref{prop:mom_polytopes} that $\{\rho_1^a, \rho_2^a\}$ is one of the sets in \eqref{eq:alpha_cone_1}, \eqref{eq:alpha_cone_2} or \eqref{eq:alpha_cone_3}. Let $M_1$ be the corresponding smooth affine spherical $\GL(2)$-variety in \cref{table_gl2_pointed_wl}, that is, $M_1 = \GL(2) \times _{\TC} \C_{-(j\alpha+ \varepsilon_1)}$ when $\{\rho_1^a, \rho_2^a\}$ is the set in \eqref{eq:alpha_cone_1}, $M_1 = \GL(2) \times _{\TC} \C_{-(j\alpha- \varepsilon_2)} $ when $\{\rho_1^a, \rho_2^a\}$ is the set in \eqref{eq:alpha_cone_2} and $M_1 = \GL(2)\left/\left\lbrace \begin{pmatrix} z^j & 0 \\ 0 & z^{j+1}\end{pmatrix}: z \in \C^{\times}\right\rbrace \right.$ when $\{\rho_1^a, \rho_2^a\}$ is the set in \eqref{eq:alpha_cone_3}. As in \cite[\S 4.1]{sjamaar-convexreex}, we view $M_1$ as a Hamiltonian $\U(2)$-manifold by embedding it into a unitary representation of $\U(2)$. Let $\Psi : M_1 \to \ft_+$ be the corresponding invariant momentum map. It follows from \cite[Theorem 4.9]{sjamaar-convexreex} that $\Psi(M_1) = \cone \{\rho_1^a, \rho_2^a\}$. By translating the momentum map of $M_1$ by $a$, we ensure that $\Psi(M_1) = a + \cone \{\rho_1^a, \rho_2^a\} = a + \cone (\mop -a)$, in other words, that $\Psi(M_1)$ is equal to $\mop$ in a neighborhood of $a$. One can now apply non-abelian symplectic cutting to $M_1$ to obtain a multiplicity free $\U(2)$-manifold with momentum polytope $\mop$. By Knop's uniqueness result in \cref{thm:knop_classif_mf}, this multiplicity free manifold has to be $M$. Because, as mentioned in \cref{rem:resultkaehler}, non-abelian symplectic cutting is a local construction which cannot be realized by ``global'' symplectic reduction, this does not yet guarantee that $M$ is K\"ahler, even though $M_1$ was. 

Nevertheless, under the assumptions of the current lemma, \cite[Corollary 4]{martens-thaddeus-nonabsc} yields that $M$ can be constructed as the  symplectic reduction of a symplectic $\U(2)$-manifold $\widetilde{M}$, which is K\"ahler because $M_1$ is. It then follows that $M$ is K\"ahler by the general properties of symplectic reduction.  The key point which allows us to apply \loccit\ is that 
\begin{equation}
\label{eq:intersQ}
\mop = \Psi(M_1) \cap \mathsf{Q},
\end{equation}
where $\mathsf{Q}$ is an \emph{outward-positive} polyhedral set, using the terminology of \cite[Definition 3]{martens-thaddeus-nonabsc}. To describe $\mathsf{Q}$, let $\mathsf{n}_1, \mathsf{n_2}, \ldots, \mathsf{n}_r$ be inward-pointing normal vectors to the $r$ edges of $\mathcal{P}$ that do not contain the vertex $a$ of $\mop$ lying on the wall of $\wc$. By assumption
\begin{equation} \label{eq:outward_pos}
\<\alpha^\vee, \mathsf{n_i}\> \leq 0 \quad \text{for all $i \in \{1,2,\ldots,r\}$}.
\end{equation} 
Let $\eta_1, \eta_2, \ldots, \eta_r \in \R$ be such that, for each $i \in \{1,2,\ldots,r\}$, 
\[ \mop \cap \{v \in \wc: \<v,\mathsf{n}_i\> = \eta_i\}\]
is the edge of $\mop$ to which $\mathsf{n}_i$ is an inward-pointing normal. Now we set 
\[\mathsf{Q} = \{v \in \ft_+: \<v,\mathsf{n}_i\> \ge \eta_i\text{ for all $i \in \{1,2,\ldots,r\}$}\}.\]
Then \eqref{eq:intersQ} holds, and \eqref{eq:outward_pos} says precisely that $\mathsf{Q}$ is outward-positive. As explained in \cite[Remark 2]{martens-thaddeus-nonabsc}, one may need to (and can) impose some extra inequalities to make $\mathsf{Q}$ \emph{universal} in the parlance of \cite[Definition 2]{martens-thaddeus-nonabsc}. 
\end{proof}
\begin{remark}\label{rem:pos_edge}
In the situation of \cref{lem:not_too_many_pos_invcompl}, if $\mop$ has a positive edge not containing $a$, then it has one that is adjacent to an edge containing $a$. This follows from the
convexity of $\mop$.
\end{remark}
For the final step in the proof of \cref{prop:invcompl}, we will make use of work of S.~Tolman's. 
In \cite{tolman-exnonkaehlertorusactions}, she constructed an example of a Hamiltonian $T$-space of complexity one in dimension six that does not
admit a $T$-invariant compatible complex structure. She proved this by checking that her example does not satisfy a certain \emph{extension criterion}
and showing that this criterion is necessary  for a $T$-invariant compatible complex structure to exist. For the convenience of the reader,
we will recall this criterion here together with the definitions necessary to formulate it. The criterion applies to compact Hamiltonian $\U(1)^n$-manifolds for any $n \in \N$, but to avoid introducing additional notation, we will use $T$ for the acting torus, as this is the setting where we will apply it. We refer to \cite[\S\S 2,3]{tolman-exnonkaehlertorusactions} for details.

By the \textbf{x-ray} of $M$, we mean its orbit type stratification
\[
\mathcal{X}=\bigcup_{H \text{ subgroup of }T} \{\text{connected components of }M^H \}
\]
together with the convex polytopes that are the images of its elements
under the $T$-momentum map $\mu_T$. We say that a convex polytope $\Delta\subset \ft^*$ (resp. a strictly convex cone $C\subset \ft^*$) is \textbf{compatible} with this x-ray if
(there exists a neighborhood $U$ of the vertex of $C$ such that)
for each face $\sigma$ of $\Delta$ (resp. $C$), we can choose $X_{\sigma}\in \mathcal{X}$ such that 
\begin{align}
& \dim(\mu_T(X_{\sigma}))=\dim(\sigma),\\ 
& \sigma \subset \mu_T(X_{\sigma}) \text{ (resp.\  $\sigma \cap U \subset \mu_T(X_{\sigma})$), and }\\
& \text{$X_{\sigma} \subset X_{\sigma'}$ whenever $\sigma$ and $\sigma'$ are faces of $\Delta$ (resp. $C$) with $\sigma \subset \sigma'$.}
\end{align}
We say that $\Delta$ is an \textbf{extension} of $C$ when there exists a neighborhood $U$ of the vertex of $C$ with $C\cap U=\Delta \cap U$.
\begin{definition}
	An x-ray satisfies the \textbf{extension criterion} if every compatible strictly convex cone admits an extension to a compatible convex polytope.
\end{definition}
\begin{theorem}[{\cite[Theorem 3.3]{tolman-exnonkaehlertorusactions}}]
\label{thm:extcrit}
	Let $M$ be a Hamiltonian $T$-manifold that does not satisfy the extension criterion. Then $M$ does not admit a $T$-invariant compatible
	complex structure.
\end{theorem}
It is clear that a compatible convex cone of dimension one always admits an extension to a compatible convex polytope, so this criterion
only needs to be checked for compatible strictly convex cones of dimension at least two. On the other hand, since our $\mu_T$ takes values in a vector space of dimension 2,
we only need to check this condition for compatible strictly convex cones of dimension exactly two. The vertices of those have to be image of a $T$-fixed point $p$,
and the edges locally have to be images of weight spaces of the corresponding isotropy representation of $T$ at $p$. Consequently,
we can describe every such cone by giving two line segments starting at $\mu_T(p)$ that correspond to linearly independent weights of the isotropy representation at $p$,
and any such line segment can be described by a pair of points in $\mop_T = \mu_T(M)$, one of which is $\mu_T(p)$. We will use this identification throughout.

Combinatorially linking Tolman's criterion to the K\"ahlerizability criterion of \cref{prop:invcompl} we now also extend \cite[Theorem 9.1]{woodward-kaehler} of Woodward's about certain multiplicity free $\SO(5)$-manifolds to multiplicity free $\U(2)$-manifolds with trivial principal isotropy whose momentum polytope intersects the Weyl wall in one point. We also rephrase our K\"ahlerizability criterion in terms of the $T$-momentum polytope and the $T$-fixpoints. 
\begin{proposition} \label{prop:atiyah} 
Let $M$ be a multiplicity free $\U(2)$-manifold with trivial principal isotropy group whose momentum polytope $\mop$ intersects the Weyl wall
at exactly one point. Then the following are equivalent:
\begin{enumerate}[(1)]
\item \label{atiyah4} Every positive edge of $\mop$ contains the vertex $a$ of $\mop$ lying on the Weyl wall. 
	\item \label{atiyah5} $M$ admits a $\U(2)$-invariant compatible complex structure.
		\item \label{atiyah3} $M$ admits a $T$-invariant compatible complex structure.
	\item \label{atiyah1} The x-ray of $M$ satisfies the extension criterion.
	\item \label{atiyah2} The set $M^T$ is mapped to the boundary of $\mop_T = \mu_T(M)$ under the $T$-momentum map $\mu_T$.
\end{enumerate}
\end{proposition}

\begin{proof}
First we show that (\ref{atiyah2}) implies (\ref{atiyah4}). If (\ref{atiyah4}) does not hold, then $\mop$ contains a positive edge that does not meet the wall of the Weyl chamber. This means that the two vertices $\mathsf{v}$ and $\mathsf{w}$ adjacent to this edge are the images under $\mu$ as well as under $\mu_T$ of $T$-fixed points in $M$, see \cref{thm:sjamaar_mp}(\ref{item_pol_vertex_interior}).  Call the vertex of $\mop$ on the wall of the Weyl chamber $v_0$, and set $\mathsf{v}'= s_{\alpha}(\mathsf{v})$,  $\mathsf{w}'= s_{\alpha}(\mathsf{w})$. Then it follows from \cref{prop:convexhullweylorbit} that the polytope 
\begin{equation} \label{eq:auxpol}
\conv(v_0,\sv,\sv', \mathsf{w}, \mathsf{w}')
\end{equation} is a subset of $\mop_T$.   Using that the edge $(\sv,\mathsf{w})$ of $\mop$ is positive, elementary geometric considerations show that $\mathsf{v}$ or $\mathsf{w}$ lies in the  interior of  the polytope \eqref{eq:auxpol}, and therefore not on the boundary of $\mop_T$.  This shows that (\ref{atiyah2}) does not hold.

We turn to the implication ``(\ref{atiyah4}) $\Rightarrow$ (\ref{atiyah2}).'' Let $m \in M^T$. First observe that $\mu(m)  \in (\fk^*)^T = \ft^*$ by the equivariance of $\mu$, and therefore that $\mu_T(m) = \mu(m)$.  Next, $\mu_T(m) \in \{\Psi(m), s_{\alpha}(\Psi(m))\}$ thanks to the well-known isomorphism $\fk^*/K \cong \ft^*/\{e,s_{\alpha}\}$ induced by the restriction map $\fk^* \to \ft^*$. Furthermore, $\Psi(m)$ is a vertex of $\mop = \Psi(M)$. Indeed, if $\Psi(m)$ lies on the Weyl wall, this is true by assumption and if $\Psi(m)$ lies in the interior of the Weyl chamber, then it follows from \cref{thm:sjamaar_mp}(\ref{item_pol_vertex_cone}) because $\dim_{\R}T_m(K\cdot m) = 2$.   
We first consider the case that $\mu_T(m)$ lies on  the Weyl wall. Then $\mu_T(m) = v_0$ and it follows from parts (\ref{fixponts_halfrefl}) and (\ref{fixpoints_refl}) of \cref{rem:fixpoints} that  
$v_0$ is of type \eqref{eq:alpha_cone_1} or of type \eqref{eq:alpha_cone_2}. Let $(v_0,\sv)$ be the edge of $\mop$ that is perpendicular to the Weyl wall. Then, using \cref{prop:convexhullweylorbit}, one deduces that $(s_{\alpha}(\sv), \sv)$ is an edge of $\mop_T$ and therefore that $v_0 = \mu_T(m)$ lies on the boundary of $\mop_T$. 
Suppose now that (\ref{atiyah2}) does not hold and that $m \in M^T$ is such that $\mu_T(m)$ does not lie on the boundary of $\mop_T$.  As we just saw, this implies that $\mu_T(m)$ does not lie on the Weyl wall. It follows from \cref{prop:convexhullweylorbit} that the segment $[\mu_T(m), s_{\alpha}(\mu_T(m))] = [\Psi(m),s_{\alpha}(\Psi(m))]$, which is perpendicular to the Weyl wall, lies in the interior of $\mop_T$. Therefore (at least) one of the two edges of  $\mop$ adjacent to $\Psi(m)$ is positive and does not meet the Weyl wall, which means that (\ref{atiyah4}) does not hold.  We have shown that (\ref{atiyah2}) follows from (\ref{atiyah4}).

Next, we observe that (\ref{atiyah5})  follows from (\ref{atiyah4}) by \cref{lem:not_too_many_pos_invcompl}, that the implication ``(\ref{atiyah5}) $\Rightarrow$ (\ref{atiyah3})'' is trivial and that (\ref{atiyah1}) follows from (\ref{atiyah3}) by \cref{thm:extcrit}. 
	
In the remainder of the proof, we show that (\ref{atiyah1}) implies (\ref{atiyah4}). 
We label the vertices of $\mop$ clockwise from $v_0$ to $v_n$ (starting at the wall), and we denote by $v_j'$ the reflection $s_{\alpha}(v_j)$ of $v_j$ across the Weyl wall.
	Note that for $j\neq 0$ the line segment $(v_j,v_j')$ is always the image of a connected component of $M^{Z(\U(2))}$, where $Z(\U(2))$ is the center of $\U(2)$,  and that this connected component is a sphere except when
	an edge of $\mop$ is adjacent to $v_j$ is parallel to $\alpha$, in which case it is $4$-dimensional by \cref{thm:sjamaar_mp}(\ref{item_pol_vertex_cone})
	together with \cref{rem_thm_sjam_mp}(\ref{rem_thm_sjam_mp_slice}).
	
	We  consider two cases, depending on whether the vertex $v_0$ is of type \eqref{eq:alpha_cone_1} (respectively of type \eqref{eq:alpha_cone_2}, which is clearly equivalent)
	or of type \eqref{eq:alpha_cone_3}. 
	In each case, the edges of the x-ray are determined by $\mop$ in the following way:
	\begin{itemize}
		\item In the first case, there are the aforementioned connected components of $M^{Z(\U(2))}$ together with all spheres belonging to those edges $(v_j,v_{j+1})$
		and $(v_j',v_{j+1}')$ which are not parallel to $\alpha$.
		\item In the second case, there are the connected components of $M^{Z(\U(2))}$ together with all spheres belonging to those edges $(v_j,v_{j+1})$
		and $(v_j',v_{j+1}')$, $j\leq n-1$, which are not parallel to $\alpha$, and on top the spheres belonging to $(v_n,v_1')$ and $(v_1,v_n')$.
		The latter are included, because the horizontal edge starting from $v_1$, for example, needs to end in one of $v_1',\hdots,v_{n}'$,
		and due to the convexity of the reflection of $\mop$ across the Weyl wall, the only possible endpoint is then $v_n'$.
	\end{itemize}
	Assume that we are in the second case: $v_0$ is of type \eqref{eq:alpha_cone_3}.  Suppose that (\ref{atiyah4}) does not hold (a polytope illustrating this situation can be found on the right in \cref{Woodward}). By \cref{rem:pos_edge} and without loss of generality, we may assume that the positive edge
	is the edge $(v_1,v_2)$. Then $v_1$ and $v_1'$ are in the interior of $\mop_T$. We claim that the  compatible cone determined by the pair of edges
	$(v_1,v_1'), (v_1,v_n')$ does not admit an extension to a compatible convex polytope. Indeed, the edge $(v_1',v_n)$ emerging from $v_n$
	cannot be part of such a polytope (since this edge intersects the edge $(v_1,v_n')$ in a point which is not the image of a $T$-fixpoint) and neither can the edge
	$(v_1',v_2')$ as convexity would not hold.	
\begin{figure}[ht]
	\begin{tikzpicture}
		\draw[step=1, dotted, gray] (0,0) grid (6,6);

		\draw[very thick] (2,1) -- (1,2);
		\draw[very thick] (2,1) -- (2,5);
		\draw[very thick] (5,2) -- (1,2);
		\draw[very thick] (1,5) -- (1,2);
		\draw[very thick] (2,1) -- (5,1);
		\draw[very thick] (5,2) -- (5,1);
		\draw[very thick] (2,5) -- (1,5);
		\draw[very thick] (2,5) -- (5,2);
		\draw[very thick] (1,5) -- (5,1);
		\draw[very thick, dashed] (0,0)--(6,6);
		\draw[fill] (2,2) circle(0pt) node[below right]{$v_0$};
		\draw[fill] (2,1) circle(2pt) node[below]{$v_3$};
		\draw[fill] (5,1) circle(2pt) node[below]{$v_2$};
		\draw[fill] (5,2) circle(2pt) node[right]{$v_1$};
		\draw[fill] (1,2) circle(2pt) node[left]{$v_3'$};
		\draw[fill] (2,5) circle(2pt) node[above]{$v_1'$};
		\draw[fill] (1,5) circle(2pt) node[left]{$v_2'$};
	\end{tikzpicture}
	\begin{tikzpicture}
	\draw[step=1, dotted, gray] (0,0) grid (6,6);

	\draw[very thick] (2,1) -- (1,2);
	\draw[very thick] (2,1) -- (2,3);
	\draw[very thick] (3,2) -- (1,2);
	\draw[very thick] (1,5) -- (1,2);
	\draw[very thick] (2,1) -- (5,1);
	\draw[very thick] (1,5) -- (5,1);
	\draw[very thick] (3,2) -- (2,3);
	\draw[very thick] (2,3) -- (1,5);
	\draw[very thick] (3,2) -- (5,1);
	\draw[very thick, dashed] (0,0)--(6,6);
	\draw[fill] (2,1) circle(2pt) node[below]{$v_3$};
	\draw[fill] (2,2) circle(0pt) node[below right]{$v_0$};
	\draw[fill] (5,1) circle(2pt) node[below]{$v_2$};
	\draw[fill] (3,2) circle(2pt) node[right]{$v_1$};
	\draw[fill] (1,2) circle(2pt) node[left]{$v_3'$};
	\draw[fill] (2,3) circle(2pt) node[above]{$v_1'$};
	\draw[fill] (1,5) circle(2pt) node[left]{$v_2'$};
\end{tikzpicture}
	\caption{Two polytopes $\mop$ with $v_0$ of type \eqref{eq:alpha_cone_3} and their x-rays. The left x-ray satisfies the extension criterion, the right one does not. 
	\label{Woodward}}
\end{figure}
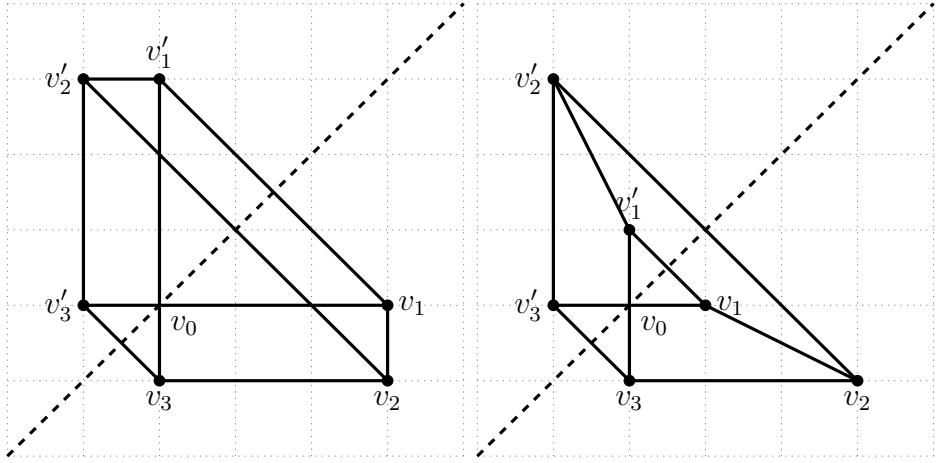
	
	Now assume that we are in the second case:  $v_0$ is of type \eqref{eq:alpha_cone_1}. 
	Suppose again that (\ref{atiyah4}) does not hold (such a polytope is given on the right in \cref{fig:morexrays}).  
		One of the compatible cones with vertex $v_0$ is spanned by the pairs of edges 
	$(v_0,v_1'), (v_0, v_n)$.
If it admitted an extension to a compatible
		convex polytope $\mathcal{Q}$, then the second edge of $\mathcal{Q}$ adjacent to $v_1'$ would have to be $(v_1',v_1)$, and
		the second edge emerging from $v_1$ would have to be $(v_1,v_2)$, which contradicts the convexity of $\mathcal{Q}$ because (\ref{atiyah4}) does not hold.  
\begin{figure}[ht] 
	\begin{tikzpicture}
		\draw[step=1, dotted, gray] (0,0) grid (5,5);

		\draw[very thick, dashed] (0,0)--(5,5);
		\draw[fill] (2,2) circle(2pt) node[below]{$v_0$};
		\draw[fill] (5,2) circle(2pt) node[right]{$v_1$};
		\draw[fill] (5,0) circle(2pt) node[right]{$v_2$};
		\draw[fill] (4,0) circle(2pt) node[left]{$v_3$};
		\draw[fill] (2,5) circle(2pt) node[above]{$v_1'$};
		\draw[fill] (0,5) circle(2pt) node[above]{$v_2'$};
		\draw[fill] (0,4) circle(2pt) node[below]{$v_3'$};
		
		\draw[very thick] (2,2) -- (5,2) -- (5,0) -- (4,0) -- (2,2);
		\draw[very thick] (2,2) -- (2,5) -- (0,5) -- (0,4) -- (2,2);
		\draw[very thick] (5,2) -- (2,5);
		\draw[very thick] (5,0) -- (0,5);
	\end{tikzpicture}
	\begin{tikzpicture}
		\draw[step=1, dotted, gray] (0,0) grid (5,5);

		\draw[very thick, dashed] (0,0)--(5,5);
		\draw[fill] (2,2) circle(2pt) node[below]{$v_0$};
		\draw[fill] (3,2) circle(2pt) node[below]{$v_1$};
		\draw[fill] (5,1) circle(2pt) node[right]{$v_2$};
		\draw[fill] (3,1) circle(2pt) node[left]{$v_3$};
		\draw[fill] (2,3) circle(2pt) node[left]{$v_1'$};
		\draw[fill] (1,5) circle(2pt) node[above]{$v_2'$};
		\draw[fill] (1,3) circle(2pt) node[left]{$v_3'$};
		
		\draw[very thick] (2,2) -- (3,2) -- (5,1) -- (3,1) -- (2,2);
		\draw[very thick] (2,2) -- (2,3) -- (1,5) -- (1,3) -- (2,2);
		\draw[very thick] (3,2) -- (2,3);
		\draw[very thick] (5,1) -- (1,5);
	\end{tikzpicture}
	\caption{Two polytopes $\mop$ with $v_0$ of type \eqref{eq:alpha_cone_1} and their x-rays. The left x-ray satisfies the extension criterion, the right one does not. \label{fig:morexrays}}
\end{figure}
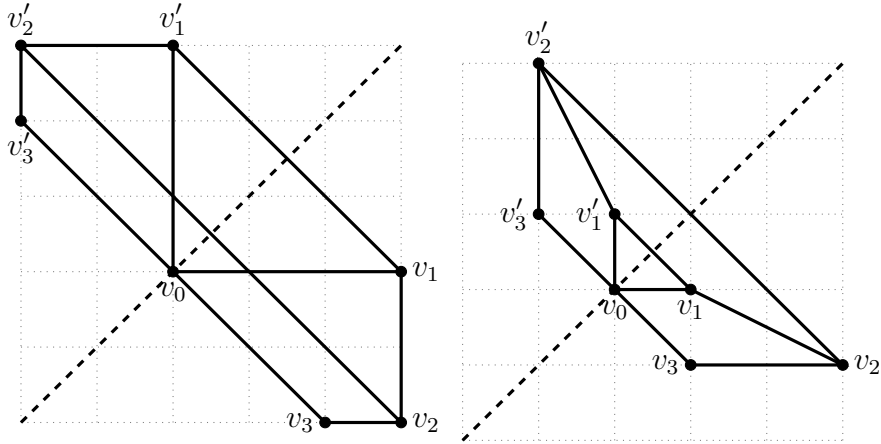 
\end{proof}

\begin{remark}
The momentum polytopes in \cref{fig:triangle_edge_on_wall,fig:triangle_Delzant}  show that the equivalence of (\ref{atiyah2}) and (\ref{atiyah4}) of \cref{prop:atiyah} do not hold  when $\mop$ does not meet the Weyl wall in exactly one point. On the other hand, \cref{prop:invcompl} tells us that when $\mop(M)$ does not meet the Weyl wall in exactly one point, then the multiplicity free $\U(2)$-manifold $M$ with trivial principal isotropy group always admits an invariant compatible complex structure, and therefore the equivalences (\ref{atiyah4}) $\Leftrightarrow$ (\ref{atiyah3}) $\Leftrightarrow$ (\ref{atiyah1}) hold in this case as well. 
\end{remark}

\begin{proof}[Proof of \cref{prop:invcompl}] \label{proof_prop_invcompl}
The proposition immediately follows from \cref{lem:nosvonwallkaehler} and the equivalence of
(\ref{atiyah4}) and (\ref{atiyah5})  
in \cref{prop:atiyah}.
\end{proof}

\section{Triangles} \label{sec:triangles}
We continue to use the notation in \cref{ex:u2}. In this section, we classify the multiplicity free $\U(2)$-manifolds with trivial principal isotropy group of which the momentum polytope is a triangle. The following lemma determines the \emph{Delzant} triangles and will be used to describe the triangles in $\ft_+$ which can occur as momentum polytopes of such manifolds.

\begin{lemma}\label{lemma:delzant_triangle}
Let $\mathsf{u},\mathsf{v}$ and $\mathsf{w} \in \ft^* = \wl \otimes_{\Z}\R$ such that 
\[\mathsf{v} - \mathsf{u}, \mathsf{w}-\mathsf{u}, \mathsf{w}-\mathsf{v} \in \wl \otimes_{\Z}\Q.\]and let $\rho_1,\rho_2,\rho_3$ be the primitive elements of $\wl$ such that
\[\Rp\rho_1 = \Rp(\mathsf{v}-\mathsf{u}), \Rp\rho_2 = \Rp(\mathsf{w}-\mathsf{u}), \Rp\rho_3 = \Rp(\mathsf{w}-\mathsf{v}).\]
Suppose $(\rho_1,\rho_2)$ is a basis of $\wl$. Both pairs $(\rho_1,\rho_3)$ and $(\rho_2,\rho_3)$ are bases of $\wl$ if and only if $\rho_3 = \rho_2 - \rho_1$. 
\end{lemma}
\begin{proof}
The ``if'' statement is clear. To prove the converse, let $a,b\in \Z$ such that $\rho_3 = a\rho_1 + b\rho_2$. It follows from the assumption that $(\rho_1,\rho_3)$ and $(\rho_2,\rho_3)$ are bases of $\wl$, that $a,b \in \{1,-1\}$. The definitions of  $\rho_1,\rho_2,\rho_3$ then imply that $a=-1$ and $b=1$.  
\end{proof} 

Recall from \cref{ex:u2} that $\alpha = \varepsilon_1 - \varepsilon_2$ is the simple root of $\U(2)$. 
\begin{proposition} \label{prop:triangles}
The triangles in $\ft_+$ that occur as momentum polytopes of multiplicity free $\U(2)$-manifolds with trivial principal isotropy group are:
\begin{enumerate}
\item \label{triangles_delzant} $r(-\varepsilon_2)  + s(\varepsilon_1+\varepsilon_2) + t\cdot \conv(0,a_1(-\varepsilon_2) + b_1\varepsilon_1,a_2(-\varepsilon_2) + b_2\varepsilon_1)$,  where $r,t \in \R_{>0}$, $s \in \R$, $a_1,b_1,a_2,b_2\in \Z$ with $\det\begin{pmatrix}
a_1 & a_2 \\ b_1 & b_2
\end{pmatrix} =  1$ and $a_i+b_i \geq 0$ for each $i \in \{1,2\}$;
\item \label{triangles_side_wall}  $s(\varepsilon_1+\varepsilon_2) + t\cdot\conv(0,\ell(\varepsilon_1 + \varepsilon_2), k(\varepsilon_1 + \varepsilon_2) + \varepsilon_1)$, where $s \in \R$, $t \in \R_{>0}$,  $k \in \Z$, $\ell \in \{-1,1\}$;
\item \label{triangles_halfright_1} $s(\varepsilon_1+\varepsilon_2) + t\cdot\conv(0,\alpha, j\alpha+\varepsilon_1)$, where $s \in \R$, $t \in \R_{>0}$, $j \in \N$
\item \label{triangles_halfright_2} $s(\varepsilon_1+\varepsilon_2) + t\cdot\conv(0,\alpha, j\alpha-\varepsilon_2)$, where $s \in \R$, $t \in \R_{>0}$, $j \in \N$;
\item \label{triangles_right} $s(\varepsilon_1+\varepsilon_2) + t\cdot\conv(0,\varepsilon_1, -\varepsilon_2)$, where $s \in \R$, $t \in \R_{>0}$. 
\end{enumerate}
\end{proposition}
\begin{proof}
Observe that the all the sets in \cref{eq:cone_1,eq:cone_2,eq:alpha_cone_1,eq:alpha_cone_2} and the one in \cref{eq:alpha_cone_3} with $j=0$ are bases of $\wl$.    
The proposition now follows from \cref{prop:mom_polytopes} and \cref{lemma:delzant_triangle} once we prove the following claim: if $\Pc \subset \ft_+$ is a triangle satisfying conditions (\ref{item:mom_poly_2}) and (\ref{item:mom_poly_3}) of \cref{prop:mom_polytopes} and $a$ is a vertex of $\Pc$ such that $\<\alpha^{\vee},a\>=0$ and
\[\{\rho_1^a, \rho_2^a\} = \{j\alpha+\varepsilon_1, j\alpha-\varepsilon_2\}  \quad \text{ for some }j \in \N, \]
then $j=0$. We may assume that $\rho_1^a = j\alpha+\varepsilon_1$ and $\rho_2^a = j\alpha-\varepsilon_2$. Let $b,c$ be the other two vertices of $\Pc$, such that 
$\Rp(b-a) = \Rp\rho_1^a$ and $\Rp(c-a) = \Rp\rho_2^a$ and write $\rho$ for the primitive element of $\wl$ on the ray $\Rp (c-b)$. Since $(\rho,\rho_1^a)$ is a basis of $\wl$ by condition (\ref{item:mom_poly_3}) of \cref{prop:mom_polytopes}, there exist $m,n \in \Z$ such that $\rho_2^a = m\rho + n \rho_1^a$. Using that $(\rho,\rho_2^a)$ is also a basis of $\wl$ it follows that $n \in \{-1,1\}$.
As $\rho_2^a$ belongs to the cone $\{p \rho + q\rho^a_1 \colon p,q \in \Rp\}$ we obtain $n=1$ and $m> 0$. Consequently $m\rho = \rho_2^a - \rho_1^a =-(\varepsilon_1 + \varepsilon_2).$ As $\rho \in \wl$ it follows that $\rho = -(\varepsilon_1 + \varepsilon_2)$. Using once more that $(\rho,\rho_1^a)$ is a basis of $\wl$ it follows that $j=0$, which completes the  proof  of the claim. 
\end{proof}

For each triangle in \cref{prop:triangles}, Knop's \cref{thm:knop_classif_mf} guarantees the existence of a multiplicity free Hamiltonian $\U(2)$-manifold whose momentum polytope is that triangle. \Cref{thm:triangles_realisation} below gives an explicit description of these manifolds. \Cref{prop:mfd_triangles_delzant,prop:mfd_triangles_right,prop:mfd_triangles_side_wall,prop:mfd_triangles_halfright_1,prop:mfd_triangles_halfright_2} provide the Hamiltonian structures. 

\begin{theorem}\label{thm:triangles_realisation}
Let $\Qc$ be one of the triangles listed in \cref{prop:triangles} and let $M$ be the (up to isomorphism) unique multiplicity free $\U(2)$-manifold with $\mop(M) = \Qc$ and trivial principal isotropy group. Then:
\begin{enumerate}[(a)]
\item $M$ is $\U(2)$-equivariantly \emph{diffeomorphic} to the corresponding manifold listed in the second column of \cref{table_triangle_manifolds}. 
\item $M$ is isomorphic (as a Hamiltonian $\U(2)$-manifold) to the corresponding manifold listed in the second column of \cref{table_triangle_manifolds} equipped with the Hamiltonian structure described in the proposition listed in the last column of \cref{table_triangle_manifolds}.
\item \label{thm:triangles_realisation_item_3} $M$ has an invariant compatible complex structure $J$ such that the complex manifold $(M,J)$, equipped with the action of $\GL(2)$ that is the complexification of the $\U(2)$-action, is $\GL(2)$-equivariantly biholomorphic to the corresponding $\GL(2)$-variety listed in the third column  of \cref{table_triangle_manifolds}.
\end{enumerate}
\end{theorem}
\begin{proof}
In each proposition listed in the fourth column of \cref{table_triangle_manifolds}, we define a structure of multiplicity free $\U(2)$-manifold on the smooth complex $\GL(2)$-variety $M$ listed in the third column 
such that 
\begin{itemize}
\item the momentum polytope of $M$  is the corresponding triangle of \cref{prop:triangles};
\item the $\U(2)$-invariant complex structure that $M$ carries by virtue of being a smooth complex $\GL(2)$-variety (cf.\ \cref{prop:analytic}) is compatible with the symplectic form on $M$; and
\item $M$, viewed as a differentiable manifold, is $\U(2)$-equivariantly diffeomorphic to the manifold listed in the second column of \cref{table_triangle_manifolds}.
\end{itemize}
Since, in each case, the $\U(2)$-action on $M$ has a trivial principal isotropy group, all the assertions now follow from part (\ref{thm:knop_classif_mf_u_item}) of Knop's \cref{thm:knop_classif_mf} and, for assertion (\ref{thm:triangles_realisation_item_3}),  also from basic facts about the complexification of actions (see, e.g., \cite[\S 1.4]{heinzner-GIT_on_stein}).
\end{proof}

\begin{longtable}{|c|l|c|c|}
\caption{Multiplicity free $\U(2)$-manifolds $M$ with trivial principal isotropy group for which $\Pc(M)$ is a triangle, as asserted in \cref{thm:triangles_realisation}. The cases are numbered as in \cref{prop:triangles}.}\label{table_triangle_manifolds}\\
\hline
Case &  $M$ as $\U(2)$-manifold &$M$ as $\GL(2)$-variety & Prop.  \\
\hline
(\ref{triangles_delzant})& \begin{tabular}{l}$ \hfb{\U(2)}{T} {\PP(V)}$,\\ where $V=\C\oplus \C_{-\delta_1} \oplus \C_{-\delta_2}$,\\ with $\delta_1 = a_1(-\eps_2)+b_1\eps_1$,\\ $\delta_2 = a_2(-\eps_2)+b_2\eps_1$,\\
$a_1,b_1,a_2,b_2 \in \Z$ as in \\ \cref{prop:triangles}(\ref{triangles_delzant})
\end{tabular} & $\hfb{\GL(2)}{B^{-}}{\PP(V)}$ & \ref{prop:mfd_triangles_delzant}\\
\hline
(\ref{triangles_side_wall}) &  \begin{tabular}{l}$\PP\bigl((\C^2 \otimes \C_{\det^{-(k+1)}}) \oplus \C_{\det^{-\ell}} \oplus \C\bigr),$ \\
where $k \in \Z, \ell \in \{-1,1\}.$
\end{tabular} & idem & \ref{prop:mfd_triangles_side_wall}\\
\hline
(\ref{triangles_halfright_1})   & \begin{tabular}{l}$\hfb{\U(2)}{T}{\PP(\C^2 \oplus \C_{-j\alpha})},$ \\where $j \in \N.$ \end{tabular} &$\hfb{\GL(2)}{B^{-}}{\PP(\C^2 \oplus \C_{-j\alpha})}$& \ref{prop:mfd_triangles_halfright_1}\\
\hline
(\ref{triangles_halfright_2})& \begin{tabular}{l}$ \hfb{\U(2)}{T}{\PP((\C^2)^* \oplus \C_{-j\alpha})},$  \\where $j \in \N.$ \end{tabular}&$\hfb{\GL(2)}{B^{-}}{\PP((\C^2)^* \oplus \C_{-j\alpha})}$ & \ref{prop:mfd_triangles_halfright_2} \\
\hline
(\ref{triangles_right})& \begin{tabular}{l} $\SO(5)/[\SO(2)\times \SO(3)]$, \\
where $\U(2)$ acts through\\ $\U(2) \into \SO(4) \subset \SO(5)$ \end{tabular}& \begin{tabular}{l} $\SO(5,\C)/P$, \\where $P \subset \SO(5,\C)$ is the \\ minimal standard parabolic\\ 
assoc.\ to the short simple root, \\
$\GL(2)$ acts through \\ $\GL(2) \into \SO(4,\C) \subset \SO(5,\C)$.   \end{tabular} & \ref{prop:mfd_triangles_right} \\
\hline
\multicolumn{4}{|l|}{$\C$ always stands for the trivial representation; $B^-$ is defined in \cref{eq:defB-}.}
\\\multicolumn{4}{|l|}{$\C^2$ stands for the defining representation of $\GL(2)$ or its restriction to $\U(2)$, $T$ or $B^-$.}\\
\hline
\end{longtable}

We will make use of the following standard fact, which follows directly from the definitions, taking into account that $\varepsilon_1 + \varepsilon_2 \in \ft^* \subset \fu(2)^*$ is a fixpoint for the coadjoint action of $\U(2)$. 
\begin{lemma}\label{lem:stretch_and_shift}
Let $(M, \omega_M, \mu_M)$ be a compact Hamiltonian $\U(2)$-manifold with momentum polytope $\Qc$. If $s \in \R, t \in \R_{>0}$, then 
\[\mu_M^{s,t} := t\mu_M + s(\eps_1 + \eps_2)\] is a momentum map for the symplectic form $t\omega_M$ on $M$ and the momentum polytope of the Hamiltonian $\U(2)$-manifold $(M,t\om_M,\mu_M^{s,t})$ is  \[s(\varepsilon_1+\varepsilon_2) + t\cdot\Qc.\] Furthermore, if  $(M, \omega_M, \mu_M)$ is multiplicity free, then so is $(M,t\om_M,\mu_M^{s,t})$
\end{lemma}

The next proposition gives  the multiplicity free $\U(2)$-manifold associated to the momentum polytope (\ref{triangles_side_wall}) of \cref{prop:triangles}. In what follows, we will write $e_1, e_2$ for the standard basis of $\C^2$. 
\begin{proposition} \label{prop:mfd_triangles_side_wall}
Let $k \in \Z$, $\ell \in \{-1,1\}$. Let $V$ be the $\U(2)$-representation
\[V := (\C^2 \otimes \C_{\det^{-(k+1)}}) \oplus \C_{\det^{-\ell}} \oplus \C \cong V(\varpi_1 - (k+1)\varpi_2) \oplus V(-\ell \varpi_2)\oplus V(0).\] 
\begin{enumerate}[(a)]
\item  \label{prop:mfd_triangles_side_wall_part_a} The projective space $\PP(V)$, equipped with the Fubini-Study symplectic form and the momentum map $\mupv$ of \cref{ex:mupv}, is a multiplicity free $U(2)$-manifold with trivial principal isotropy group. 
\item \label{prop:mfd_triangles_side_wall_part_b}
The $T$-fixpoints in $\PP(V)$  
are 
\begin{align*}
x_1:=[(e_1\otimes 1)\oplus 0 \oplus 0],\  x_2:=[(e_2\otimes 1)\oplus 0 \oplus 0],\ x_3:=[0\oplus 1 \oplus 0],\ x_4:=[0\oplus 0\oplus 1] 
\end{align*}
and their images under $\mupv$ are (in the same order)
\begin{equation}  \label{prop:mfd_triangles_side_wall_tweights}
k(\varepsilon_1+\varepsilon_2) + \varepsilon_2,\ k(\varepsilon_1+\varepsilon_2) + \varepsilon_1,\ \ell (\varepsilon_1 + \varepsilon_2),\ 0
 \end{equation}

\item \label{prop:mfd_triangles_side_wall_part_c} The  momentum polytope of $(\PP(V),\mupv)$ is the triangle  \[\conv(0,\ell(\varepsilon_1 + \varepsilon_2), k(\varepsilon_1 + \varepsilon_2) + \varepsilon_1)\] in case (\ref{triangles_side_wall}) of \cref{prop:triangles}.
\item \label{prop:mfd_triangles_side_wall_part_d} If $s \in \R, t \in \R_{>0}$, then 
\[\mupv^{s,t} := t\mupv + s(\eps_1 + \eps_2)\] is a momentum map for the symplectic form $t\ompv$ on $\PP(V)$ and the momentum polytope of the multiplicity free $\U(2)$-manifold $(\PP(V),\mupv^{s,t})$ is the triangle  \[s(\varepsilon_1+\varepsilon_2) + t\cdot\conv(0,\ell(\varepsilon_1 + \varepsilon_2), k(\varepsilon_1 + \varepsilon_2) + \varepsilon_1)\] of case (\ref{triangles_side_wall}) in  \cref{prop:triangles} 
\end{enumerate}
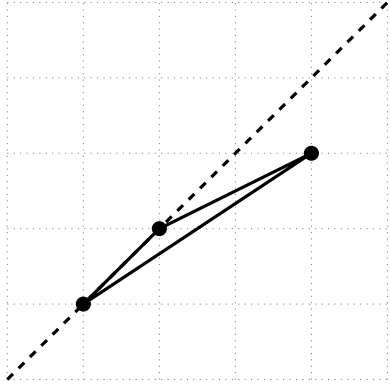
\begin{figure}[ht]
\begin{tikzpicture}
\draw[step=1, dotted, gray] (-1,-1) grid (4,4);

\draw[very thick] (0,0) -- (1,1) -- (3,2) -- (0,0);
\draw[very thick, dashed] (-1,-1)--(4,4);
  \node at (1,1)[circle,fill,inner sep=2pt]{};
  \node at (3,2)[circle,fill,inner sep=2pt]{};
 \node at (0,0)[circle,fill,inner sep=2pt]{};

\end{tikzpicture}
\caption{The triangle in part (\ref{prop:mfd_triangles_side_wall_part_d}) of \cref{prop:mfd_triangles_side_wall} for $k=2$, $\ell=1$. \label{fig:triangle_edge_on_wall}}
\end{figure}
\end{proposition}
\begin{proof}
We begin with (\ref{prop:mfd_triangles_side_wall_part_a}). It follows from \cref{ex:mupv} that $(\PP(V),\mupv)$ is a Hamiltonian $\U(2)$-manifold. A computation shows that the only element of $\U(2)$ that fixes $$[(e_1 \otimes 1) \oplus 1 \oplus 1] \in \PP(V)$$ is the identity, which implies that the principal isotropy group of the $U(2)$-action on $\PP(V)$ is trivial. Since $\PP(V)$ is compact and connected, it follows from  \cref{eq:dim_mf} that $(\PP(V),\mupv)$ is multiplicity free.

To show (\ref{prop:mfd_triangles_side_wall_part_b}) we  first observe that all the $T$-weight spaces in $V$  have dimension $1$. This implies that the $T$-fixpoints in $\PP(V)$ are exactly the lines spanned by $T$-eigenvectors in $V$, which shows the first assertion in (\ref{prop:mfd_triangles_side_wall_part_b}). It follows that $\mu(x_i) \in \ft^* \cong (\fk^*)^T$ for every $i \in \{1,2,3,4\}$. We now use \cref{ex:mupv} to compute $\mupv(x_1)$. Let $\xi \in \ft$. Since $\mathsf{v}:=e_1\otimes 1 \oplus 0 \oplus 0 \in V$ has $T$-weight $\gamma:=\varepsilon_1-(k+1)(\varepsilon_1+\varepsilon_2)$ we have
$\xi\cdot \mathsf(v) = 2\pi\sqrt{-1}\gamma(\xi)\mathsf{v}$ which implies that
$\mupv(x_1)(\xi) = -\gamma(\xi)$, that is, $\mupv(x) = -\gamma$, as claimed. The computations of $\mupv(x_2)$, $\mupv(x_3)$ and $\mupv(x_4)$ are analogous. 

 We turn to (\ref{prop:mfd_triangles_side_wall_part_c}). 
Since  \[\mupv(x_2)=k(\varepsilon_1+\varepsilon_2)+\varepsilon_1\] is the only weight in \eqref{prop:mfd_triangles_side_wall_tweights} that belongs to the interior of $\ft_+$, it is the only vertex of $\mop(M)$ in the interior of $\ft_+$, thanks to \cref{thm:sjamaar_mp}(\ref{item_pol_vertex_interior}). 

In order to apply  part (\ref{item_pol_vertex_cone}) of \cref{thm:sjamaar_mp}, we next show that the $T$-weights in the symplectic slice $N_{x_2}$ of $\PP(V)$ at $x_2$ are 
\[\Pi_{x_2} = \{k(\varepsilon_1+\varepsilon_2)+\varepsilon_1, (k-\ell)(\varepsilon_1+\varepsilon_2) + \varepsilon_1\}.\] 
Indeed, as $\PP(V)$ comes with an invariant complex structure which is compatible with its Fubini-Study symplectic form by construction, we have the following isomorphisms of $T$-modules  
\begin{multline*}
N_{x_2} \cong T_{x_2} \PP(V) / T_{x_2} (K\cdot x_2) =T_{x_2}\PP(V) / T_{x_2} \PP(\C^2 \otimes \C_{\det^{-(k+1)}}\oplus 0 \oplus 0)\\ \cong \C_{(k-\ell)(\varepsilon_1+\varepsilon_2) + \varepsilon_1} \oplus \C_{k(\varepsilon_1+\varepsilon_2)+\varepsilon_1}.\end{multline*}
Since the extremal rays 
\[\mupv(x_2) - \Rp(k(\varepsilon_1+\varepsilon_2)+\varepsilon_1) \quad \text{ and }  \mupv(x_2) - \Rp((k-\ell)(\varepsilon_1+\varepsilon_2)+\varepsilon_1)\]
of the cone $\mupv(x_2) - \operatorname{cone}\Pi_{x_2}$
intersect the wall of  the Weyl chamber $\ft_+$ in the points $0$ and  $\ell(\varepsilon_1+\varepsilon_2)$
it  follows from  part (\ref{item_pol_vertex_cone}) of \cref{thm:sjamaar_mp} that 
\[\mop(\PP(V)) = \conv(0,\ell(\varepsilon_1+\varepsilon_2), k(\varepsilon_1+\varepsilon_2)+\varepsilon_1)\]
as claimed.

Part (\ref{prop:mfd_triangles_side_wall_part_d}) follows from part (\ref{prop:mfd_triangles_side_wall_part_c}) and \cref{lem:stretch_and_shift}. 
\end{proof}

We now describe  the multiplicity free $\U(2)$-manifold associated to the momentum polytope (\ref{triangles_halfright_1}) of \cref{prop:triangles}. Recall from \cref{ex:u2} that $\alpha = \eps_1-\eps_2$ is the simple root of $\U(2)$ and $\GL(2)$. 
\begin{proposition} \label{prop:mfd_triangles_halfright_1}
Let $j \in \N$ and set 
\[M = \hfb{\GL(2)}{B^-}{\PP(\C^2 \oplus \C_{-j\alpha})}\]
where the group $B^{-}$ of lower triangular matrices in $\GL(2)$ acts on $\PP(\C^2 \oplus \C_{-j\alpha})$ through the standard linear action of $\GL(2)$ on $\C^2$ and with weight $-j\alpha$ on the $1$-dimensional space $\C_{-j\alpha}$.
\begin{enumerate}[(a)] 
\item The map
\[\hfb{\U(2)}{T}{\PP(\C^2 \oplus \C_{-j(\eps_1-\eps_2)})} \to M, [g,[y]] \mapsto [g,[y]]\]
is a $\U(2)$-equivariant diffeomorphism.  \label{prop:mfd_triangles_halfright_1_item_U2}
\item \label{prop:mfd_triangles_halfright_1_item1} Let $V$ be the irreducible $\GL(2)$-representation with highest  weight $j\alpha$ and let $\sv \in V$ be a lowest weight vector in $V$. Then
\[\iota_M: M \to Y:= \hfb{\GL(2)}{B^-}{\PP(\C^2 \oplus V)}, [g,[u\oplus z]] \to [g,[u \oplus z\sv]]\] is a $\GL(2)$-equivariant closed embedding and 
\[\iota_Y : Y \to \PP(\C^2) \times \PP(\C^2 \oplus V), [g,[u \oplus v]] \mapsto ([ge_2],[gu \oplus gv])\] is a $\GL(2)$-equivariant isomorphism of varieties. 
\item \label{prop:mfd_triangles_halfright_1_item2} Let $\om_1$ be the Fubini-Study symplectic form on $\PP(\C^2)$ and
\[\mu_1: \PP(\C^2) \to \fu(2)^*\] the associated momentum map as in \cref{ex:mupv}, $\om_2$  the  Fubini-Study symplectic form on $\PP(\C^2\oplus V)$  and  \(\mu_2: \PP(\C^2\oplus V) \to \fu(2)^*\) the associated momentum map. If $\om_M$ is the pullback  along $\iota_Y \circ \iota_M$ of the symplectic form $\om_1 + \om_2$ on $\PP(\C^2) \times \PP(\C^2 \oplus V)$ then $\om_M$ is a symplectic form on $M$ with momentum map 
\[\mu_M = (\mu_1+\mu_2) \circ \iota_Y \circ \iota_M\]
and $(M,\mu_M)$ is a multiplicity free $\U(2)$-manifold with trivial principal isotropy group. 
\item \label{prop:mfd_triangles_halfright_1_item3} Set $n := \begin{pmatrix}
0 & 1 \\ 1 & 0
\end{pmatrix} \in \U(2)$. The $T$-fixpoints in $M$ are 
\begin{align*}
& x_1:=[e,[1:0:0]],\  x_2:=[n,[1:0:0]],\ x_3:=[e,[0:1:0]],\\ &x_4:=[n,[0:1:0]],\ x_5:=[e,[0:0:1]],\ x_6:=[n,[0:0:1]] \end{align*}
and their images under $\mu_M$ are (in the same order)
\begin{equation} \label{eq:weights_toric_fixpoints_halfright}
-\eps_1 -\eps_2,\ -\eps_1 -\eps_2,\ 
-2\eps_2,\ -2\eps_1,\  j\alpha - \eps_2,\  -j\alpha - \eps_1.
\end{equation}
\item \label{prop:mfd_triangles_halfright_1_item4} The momentum polytope of $(M, \mu_M)$ is the triangle $(-\eps_1-\eps_2)+\conv(0, \alpha, j\alpha +\eps_1)$ in case (\ref{triangles_halfright_1}) of \cref{prop:triangles}
\item \label{prop:mfd_triangles_halfright_1_item5} If $s \in \R, t \in \R_{>0}$, then 
\[\mu_M^{s,t} := t\mu_M + (s+t)(\eps_1 + \eps_2)\] is a momentum map for the symplectic form $t\omega_M$ on $M$. The momentum polytope of the multiplicity free $\U(2)$-manifold $(M,\mu_M^{s,t})$ is the triangle  \[s(\varepsilon_1+\varepsilon_2) + t\cdot\conv(0,\alpha, j\alpha+\varepsilon_1)\] of case (\ref{triangles_halfright_1}) in  \cref{prop:triangles} 
\end{enumerate}
\begin{figure}[ht]
\begin{tikzpicture}
\draw[step=1, dotted, gray] (-2,-2) grid (2,2);

\draw[very thick] (0,0) -- (1,-1) -- (1,0) -- (0,0);
\draw[very thick, dashed] (-2,-2)--(2,2);
  \node at (1,-1)[circle,fill,inner sep=2pt]{};
  \node at (1,0)[circle,fill,inner sep=2pt]{};
 \node at (0,0)[circle,fill,inner sep=2pt]{};

\end{tikzpicture}
\begin{tikzpicture}
\draw[step=1, dotted, gray] (-2,-2) grid (3,2);

\draw[very thick] (0,0) -- (1,-1) -- (2,-1) -- (0,0);
\draw[very thick, dashed] (-2,-2)--(2,2);
  \node at (1,-1)[circle,fill,inner sep=2pt]{};
  \node at (2,-1)[circle,fill,inner sep=2pt]{};
 \node at (0,0)[circle,fill,inner sep=2pt]{};

\end{tikzpicture}
\begin{tikzpicture}
\draw[step=1, dotted, gray] (-1,-4) grid (5,1);

\draw[very thick] (0,0) -- (1,-1) -- (4,-3) -- (0,0);
\draw[very thick, dashed] (-1,-1)--(1,1);
  \node at (1,-1)[circle,fill,inner sep=2pt]{};
  \node at (4,-3)[circle,fill,inner sep=2pt]{};
 \node at (0,0)[circle,fill,inner sep=2pt]{};

\end{tikzpicture}
\caption{The triangle in part (\ref{prop:mfd_triangles_halfright_1_item4}) of \cref{prop:mfd_triangles_halfright_1} for $j=0$, $j=1$  and $j=3$.}
\end{figure}
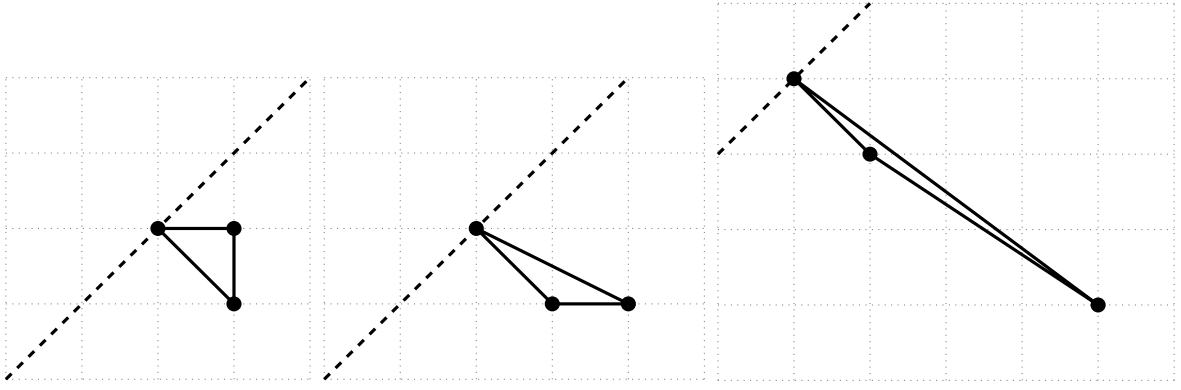
\end{proposition}

\begin{proof}
Part (\ref{prop:mfd_triangles_halfright_1_item_U2}) is just an application of \cref{prop:hfb_comparison}. 

We proceed to assertion (\ref{prop:mfd_triangles_halfright_1_item1}). The assertion about $\iota_M$ follows from the fact that $\C_{-j\alpha} \to V: z \mapsto z\mathsf{v}$ is a $B^{-}$-equivariant injective linear map. The claim about $\iota_Y$ is a standard fact; see, e.g., \cite[Lemma 2.3]{timashev-embbook}. 

The assertion in (\ref{prop:mfd_triangles_halfright_1_item2}) that $(M,\om_M, \mu_M)$ is a Hamiltonian $\U(2)$-manifold follows from standard and well-known facts about Hamiltonian actions. Furthermore, a straightforward computation shows that the isotropy group $\U(2)_x$ of (for example) $x = [e,[1:1:1]] \in M$ is trivial, which implies that the principal isotropy group is trivial as well. It now follows from  \cref{eq:dim_mf} and from the fact that $M$ is compact and connected, that $M$ is a multiplicity free. 

To prove (\ref{prop:mfd_triangles_halfright_1_item3}) we will use \cref{ex:mupv}. A straightforward calculation shows that the listed points  are the six $T$-fixpoints in $M$. It follows that their images under $\mu_M$ lie in $\ft^* \cong (\fk^*)^T$. Let $\xi \in \ft$. We begin by computing $\mu_M(x_1)(\xi)$. 
First off, \[\iota_Y \circ \iota_M (x_1) = ([e_2], [e_1 \oplus 0]).\] Since $e_2$ has weight $\eps_2$ and $e_1$ has weight $\eps_1$, we have $\xi \cdot e_2 = 2\pi\sqrt{-1}\eps_2(\xi)e_2$ and $\xi\cdot (e_1 \oplus 0) = 2\pi\sqrt{-1}\eps_1(\xi)(e_1 \oplus 0)$ which implies that 
\[\mu_1([e_2])(\xi) = -\eps_2(\xi) \text{ and }\mu_2([e_1 \oplus 0])(\xi) = -\eps_1(\xi).\]
The claimed equality $\mu_M (x_1) = -\eps_1-\eps_2$ follows. 

Similar elementary computations yield the images of $x_2$ through $x_6$ under $\mu_M$, using 
\begin{align*}
& \iota_Y \circ \iota_M(x_2) = ([e_1],[e_2\oplus 0]),\ \iota_Y \circ \iota_M(x_3) = ([e_2],[e_2\oplus 0]),\ \iota_Y \circ \iota_M(x_4) = ([e_1],[e_1\oplus 0]), \\
& \iota_Y \circ \iota_M(x_5) = ([e_2],[0\oplus \sv]),\ \iota_Y \circ \iota_M(x_6) = ([e_1],[0\oplus n  \sv]) 
\end{align*}
and
\[\xi \cdot \sv = 2\pi\sqrt{-1}(-j\alpha)(\xi)\sv,\  \xi \cdot n\sv = 2\pi\sqrt{-1}(j\alpha)(\xi)n\sv,\]
which hold because $\sv$ has weight $-j\alpha$ and $n\sv$ has weight $j\alpha$.

 We turn to (\ref{prop:mfd_triangles_halfright_1_item4}). 
Since  \[\mathsf{u}:= \mu_M(x_3)=-2\eps_2 \text{ and } \mathsf{w}:= \mu_M(x_5) = j\alpha - \eps_2\] are the only weights in \eqref{eq:weights_toric_fixpoints_halfright} that belong to the interior of $\ft_+$, they are the only possible vertices of $\mop(M)$ in the interior of $\ft_+$, thanks to \cref{thm:sjamaar_mp}(\ref{item_pol_vertex_interior}). 
In order to apply  part (\ref{item_pol_vertex_cone}) of \cref{thm:sjamaar_mp}, we next show that the $T$-weights in the symplectic slice $N_{x_3}$ of $M$ at $x_3$ are 
\[\Pi_{x_3} = \{\alpha, -j\alpha - \varepsilon_2\} \] whereas those in the symplectic slice $N_{x_5}$ at $x_5$ are
\[\Pi_{x_5} = \{j\alpha + \varepsilon_1, j\alpha + \varepsilon_2\}.\]
Indeed, as $M$ comes with an invariant complex structure which is compatible with $\omega_M$ by construction, we have the the following isomorphisms of $T$-modules
\begin{align*}N_{x_3}& \cong T_{x_3} M / T_{x_3} (K\cdot x_3) \cong T_{[0:1:0]} \PP(\C^2 \oplus \C_{-j\alpha}) \cong \C_{\varepsilon_1 - \varepsilon_2} \oplus \C_{-j\alpha - \varepsilon_2} \\
N_{x_5}& \cong T_{[0:0:1]} \PP(\C^2 \oplus \C_{-j\alpha}) \cong \C_{\varepsilon_1 +j\alpha} \oplus \C_{\varepsilon_2+j\alpha} 
\end{align*}
Since the two weights in $\Pi_{x_3}$ are linearly independent,  \cref{thm:sjamaar_mp}(\ref{item_pol_vertex_cone}) implies that $x_3$ is a vertex of $\mop(M)$, and the same holds for $x_5$. 
As \[\mathsf{w} -(j\alpha + \varepsilon_2) = \mathsf{u} \quad \text{ and } \quad \mathsf{u} - \alpha = \mathsf{w} - (j\alpha + \varepsilon_1) = -\varepsilon_1 - \varepsilon_2\] it also follows from  part (\ref{item_pol_vertex_cone}) of \cref{thm:sjamaar_mp} that $-\varepsilon_1 - \varepsilon_2$ is the only remaining vertex of $\mop(M)$ and we have proven that
\[\mop(M) = \conv(-\varepsilon_1 - \varepsilon_2, \mathsf{u}, \mathsf{w}),\]
as required.

Finally, assertion (\ref{prop:mfd_triangles_halfright_1_item5}) follows from \cref{lem:stretch_and_shift}.
\end{proof}

With proofs similar to that of \cref{prop:mfd_triangles_halfright_1}, one establishes the following descriptions of the $U(2)$-manifolds associated to the triangles (\ref{triangles_halfright_2}) and (\ref{triangles_delzant}) of \cref{prop:triangles}.

\begin{proposition} \label{prop:mfd_triangles_halfright_2}
Let $j \in \N$ and set 
\[M = \hfb{\GL(2)}{B^-}{\PP((\C^2)^* \oplus \C_{-j\alpha})}\]
where the group $B^{-}$ of lower triangular matrices in $\GL(2)$ acts on $\PP((\C^2)^* \oplus \C_{-j\alpha})$ through the linear action of $\GL(2)$ on $(\C^2)^*$ dual to the standard action on $\C^2$ and with weight $-j\alpha$ on the $1$-dimensional space $\C_{-j\alpha}$.
\begin{enumerate}[(a)] 
\item The map
\[\hfb{\U(2)}{T}{\PP((\C^2)^* \oplus \C_{-j(\eps_1-\eps_2)})} \to M, [g,[y]] \mapsto [g,[y]]\]
is a $\U(2)$-equivariant diffeomorphism.  \label{prop:mfd_triangles_halfright_2_item_U2}
\item \label{prop:mfd_triangles_halfright_2_item1} Let $V$ be the irreducible $\GL(2)$-representation with highest  weight $j\alpha$ and let $\sv \in V$ be a lowest weight vector in $V$. Then
\[j_M: M \to \PP(\C^2) \times \PP((\C^2)^* \oplus V), [g,[u\oplus z]] \mapsto ([ge_2],[gu \oplus gz\mathsf{v}])\] is a $\GL(2)$-equivariant closed embedding. 
\item \label{prop:mfd_triangles_halfright_2_item2} Let $\om_1$ be the Fubini-Study symplectic form on $\PP(\C^2)$ and
\[\mu_1: \PP(\C^2) \to \fu(2)^*\] the associated momentum map as in \cref{ex:mupv}, $\om_2$  the  Fubini-Study symplectic form on $\PP((\C^2)^*\oplus V)$  and  \(\mu_2: \PP(\C^2\oplus V) \to \fu(2)^*\) the associated momentum map. If $\om_M$ is the pullback  along $j_M$ of the symplectic form $\om_1 + \om_2$ on $\PP(\C^2) \times \PP((\C^2)^* \oplus V)$ then $\om_M$ is a symplectic form on $M$ with momentum map 
\[\mu_M = (\mu_1+\mu_2) \circ j_M\]
and $(M,\mu_M)$ is a multiplicity free $\U(2)$-manifold with trivial principal isotropy group. 
\item \label{prop:mfd_triangles_halfright_2_item4} The momentum polytope of $(M, \mu_M)$ is the triangle $\conv(0, \alpha, j\alpha -\eps_2)$ in case (\ref{triangles_halfright_2}) of \cref{prop:triangles}
\item \label{prop:mfd_triangles_halfright_2_item5} If $s \in \R, t \in \R_{>0}$, then 
\[\mu_M^{s,t} := t\mu_M + s(\eps_1 + \eps_2)\] is a momentum map for the symplectic form $t\omega_M$ on $M$. The momentum polytope of the multiplicity free $\U(2)$-manifold $(M,\mu_M^{s,t})$ is the triangle \[s(\eps_1 + \eps_2)+t\cdot\conv(0, \alpha, j\alpha -\eps_2)\] in case (\ref{triangles_halfright_2}) of \cref{prop:triangles}.
\end{enumerate}
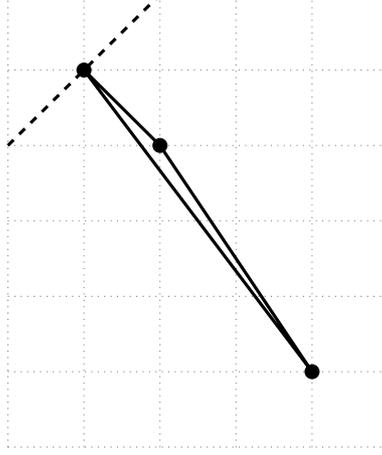
\begin{figure}[h]
\begin{tikzpicture}
\draw[step=1, dotted, gray] (-1,-5) grid (4,1);

\draw[very thick] (0,0) -- (1,-1) -- (3,-4) -- (0,0);
\draw[very thick, dashed] (-1,-1)--(1,1);
  \node at (1,-1)[circle,fill,inner sep=2pt]{};
  \node at (3,-4)[circle,fill,inner sep=2pt]{};
 \node at (0,0)[circle,fill,inner sep=2pt]{};

\end{tikzpicture}
\caption{The triangle in part (\ref{prop:mfd_triangles_halfright_2_item4}) of \cref{prop:mfd_triangles_halfright_2} for $j=3$.}
\end{figure}

\end{proposition} 

\begin{proposition} \label{prop:mfd_triangles_delzant}
Let $a_1,b_1,a_2,b_2\in \Z$ with $\det\begin{pmatrix}
a_1 & a_2 \\ b_1 & b_2
\end{pmatrix} =  1$ and $a_i+b_i \geq 0$ for each $i \in \{1,2\}$. Set $\delta_1 = a_1(-\eps_2)+b_1\eps_1$, $\delta_2 = a_2(-\eps_2)+b_2\eps_1$ and
\[M = \hfb{\GL(2)}{B^-}{\PP(\C \oplus \C_{-\delta_1} \oplus \C_{-\delta_2})}\]
where the group $B^{-}$ of lower triangular matrices in $\GL(2)$ acts on $\PP(\C \oplus \C_{-\delta_1} \oplus \C_{-\delta_2})$ through its linear action with weight $0$,  $-\delta_1$ and  $-\delta_2$ on the $1$-dimensional spaces $\C, \C_{-\delta_1}$ and  $\C_{-\delta_2}$, respectively.
\begin{enumerate}[(a)] 
\item The map
\[\hfb{\U(2)}{T}{\PP(\C \oplus \C_{-\delta_1} \oplus \C_{-\delta_2})} \to M, [g,[y]] \mapsto [g,[y]]\]
is a $\U(2)$-equivariant diffeomorphism.  \label{prop:mfd_triangles_delzant_item_U2}
\item \label{prop:mfd_triangles_delzant_item1} For $i \in \{1,2\}$, let $V_i$ be the irreducible $\GL(2)$-representation with lowest  weight $-\delta_i$ and let $\sv_i$ be a lowest weight vector in $V_i$. Then
\[j_M: M \to \PP(\C^2) \times \PP(\C \oplus V_1 \oplus V_2), [g,[z_0\oplus z_1\oplus z_2]] \mapsto ([ge_2],[z_0 \oplus gz_1\sv_1 \oplus gz_2\sv_2])\] is a $\GL(2)$-equivariant closed embedding. 
\item \label{prop:mfd_triangles_delzant_item2} Let $c \in \R_{>0}$. We  write $\om_1$ for the Fubini-Study symplectic form on $\PP(\C^2)$ and
\[\mu_1: \PP(\C^2) \to \fu(2)^*\] for the associated momentum map as in \cref{ex:mupv}, $\om_2$  for the  Fubini-Study symplectic form on $\PP(\C\oplus V_1 \oplus V_2)$  and  \(\mu_2: \PP(\C\oplus V_1 \oplus V_2) \to \fu(2)^*\) for the associated momentum map. If $\om^c_M$ is the pullback  along $j_M$ of the symplectic form $c\om_1 + \om_2$ on $\PP(\C^2) \times \PP(\C\oplus V_1 \oplus V_2)$ then $\om^c_M$ is a symplectic form on $M$ with momentum map 
\[\mu^c_M = (c\mu_1+\mu_2) \circ j_M\]
and $(M,\mu^c_M)$ is a multiplicity free $\U(2)$-manifold with trivial principal isotropy group. 
\item \label{prop:mfd_triangles_delzant_item4} The momentum polytope of $(M, \mu^c_M)$ is the triangle $c(-\eps_2)+\conv(0, \delta_1, \delta_2)$ in case (\ref{triangles_delzant}) of \cref{prop:triangles}
\item \label{prop:mfd_delzant_item5} If $s \in \R, r,t \in \R_{>0}$, then 
\[\mu_M^{r,s,t} := t\mu^{r/t}_M + s(\eps_1 + \eps_2) = (r\mu_1+t\mu_2) \circ j_M + s(\varepsilon_1+\varepsilon_2)\] is a momentum map for the symplectic form $t\omega^{r/t}_M = r\om_1 + t\om_2$ on $M$. The momentum polytope of the multiplicity free $\U(2)$-manifold $(M,\mu_M^{s,t})$ is the triangle \[r(-\varepsilon_2)  + s(\varepsilon_1+\varepsilon_2) + t\cdot \conv(0,\delta_1,\delta_2)\] of case (\ref{triangles_delzant}) of \cref{prop:triangles}.
\end{enumerate}
\begin{figure}[h]
\begin{tikzpicture}
\draw[step=1, dotted, gray] (-1,-1) grid (3,2);

\draw[very thick] (1,0) -- (2,1) -- (1,-1) -- (1,0);
\draw[very thick, dashed] (-1,-1)--(2,2);
  \node at (1,0)[circle,fill,inner sep=2pt]{};
  \node at (2,1)[circle,fill,inner sep=2pt]{};
 \node at (1,-1)[circle,fill,inner sep=2pt]{};

\end{tikzpicture}
\caption[A triangle in part (\ref{prop:mfd_triangles_delzant_item4}) of \cref{prop:mfd_triangles_delzant}] {The triangle in part (\ref{prop:mfd_triangles_delzant_item4}) of \cref{prop:mfd_triangles_delzant} for $c=1$, 
$\begin{pmatrix}
a_1 & a_2 \\ b_1 & b_2
\end{pmatrix} 
= \begin{pmatrix}
1 & -1 \\ 0 & 1
\end{pmatrix}.$}
\label{fig:triangle_Delzant}
\end{figure}
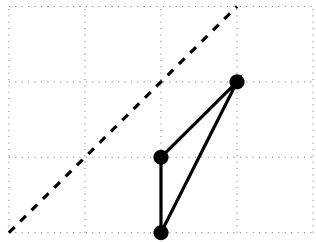
\end{proposition}

\begin{remark} The following illustrates a phenomenon already observed in \cite[Remark 4.4]{woodward-spherical} for the acting group $\SU(2)$: in contrast to the toric case (see \cite[Remark 9.5]{lerman_tolman_ham_tor_orb_tor_var}), multiplicity free manifolds for a non-abelian acting group may admit invariant compatible complex structures that are not equivariantly isomorphic. As we will now make precise, the complex manifolds in \cref{prop:mfd_triangles_halfright_1,prop:mfd_triangles_halfright_2} also carry compatible Hamiltonian structures that realize triangles in case (\ref{triangles_delzant}) of \cref{prop:triangles}. For these triangles, the corresponding multiplicity free $\U(2)$-manifold therefore admits two compatible invariant complex structures that are not equivariantly isomorphic: the one in \cref{prop:mfd_triangles_halfright_1} or \cref{prop:mfd_triangles_halfright_2} and the one in \cref{prop:mfd_triangles_delzant}.

Indeed, let $j, M, \mu_1, \mu_2, \om_1, \om_2, \iota_Y$ and $\iota_M$ be as in \cref{prop:mfd_triangles_halfright_1}. For $r,t \in \R_{>0}$ and $ s\in \R$ we set
\begin{equation} \label{eq:muMrst}
\mu_M^{r,s,t}: = [(r+t)\mu_1 + t\mu_2] \circ \iota_Y \circ \iota_M + (s+t)(\varepsilon_1 + \varepsilon_2).\end{equation}
Then $\left(M, (r+t)\om_1 + t\om_2, \mu_M^{r,s,t} \right)$ is the multiplicity free $\U(2)$-manifold with trivial principal isotropy group whose momentum polytope is the triangle 
\begin{align*}
&r(-\varepsilon_2) + s(\varepsilon_1+\varepsilon_2) + t\cdot \conv(0, \alpha, j\alpha+\varepsilon_1) \\ = \quad
&r(-\varepsilon_2) + s(\varepsilon_1+\varepsilon_2) + t\cdot \conv(0, \varepsilon_1-\varepsilon_2, (j+1)\varepsilon_1-j\varepsilon_2)
\end{align*}  in case (\ref{triangles_delzant}) of \cref{prop:triangles} 
with $\begin{pmatrix}
a_1 & a_2 \\ b_1 & b_2
\end{pmatrix}  = \begin{pmatrix}
1 & j \\ 1 & j+1
\end{pmatrix}$. Observe that, compared to the momentum map $\mu_M^{s,t}$ of \cref{prop:mfd_triangles_halfright_1}(\ref{prop:mfd_triangles_halfright_1_item5}), the `new' term $r\mu_1$ in $\mu_M^{r,s,t}$ of \cref{eq:muMrst} just causes the `old' momentum polytope to be translated by $r(-\varepsilon_2)$:
\[\mop(M, \mu_M^{r,s,t}) =\mop(M,\mu_M^{s,t}) + r(-\varepsilon_2).\] That the invariant compatible complex structure in \cref{prop:mfd_triangles_halfright_1} on this Hamiltonian manifold $(M,\mu_M^{r,s,t})$  is not $\U(2)$-equivariantly isomorphic to the one in \cref{prop:mfd_triangles_delzant} can be seen as follows. If they were, then the complex $\GL(2)$-manifolds in the two propositions would be $\GL(2)$-equivariantly isomorphic (see, e.g., \cite[\S 1.4]{heinzner-GIT_on_stein}), but this is not the case: the unique open $\GL(2)$-orbit of the complex manifold $ \hfb{\GL(2)}{B^-}{\PP(\C^2 \oplus \C_{-j\alpha})}$ from \cref{prop:mfd_triangles_halfright_1} is 
\[\GL(2)\cdot[e,[e_1 \oplus 1]] \cong\GL(2)\left/\left\lbrace \begin{pmatrix} z^j & 0 \\ 0 & z^{j+1}\end{pmatrix}: z \in \C^{\times}\right\rbrace \right.,\]
whereas the unique open $\GL(2)$-orbit of $ \hfb{\GL(2)}{B^-}{\PP(\C \oplus \C_{-\alpha} \oplus\C_{-j\alpha-\varepsilon_1})}$ from \cref{prop:mfd_triangles_delzant} is
\[\GL(2) \cdot [e,[1\oplus1 \oplus 1]] \cong \GL(2)\left/\left\lbrace \begin{pmatrix} 1 & 0 \\ a & 1\end{pmatrix}: a \in \C\right\rbrace \right..\]

Similarly, if $j, M, \mu_1, \mu_2, \om_1, \om_2, \iota_Y$ and $\iota_M$ are now as in \cref{prop:mfd_triangles_halfright_2}; $r,t \in \R_{>0}$; $ s\in \R$ and we set
\[\mu_M^{r,s,t} := [(r+t)\mu_1 + t\mu_2] \circ \iota_Y \circ \iota_M+ s(\varepsilon_1 + \varepsilon_2),\] 
then $\left(M, (r+t)\om_1 + t\om_2, \mu_M^{r,s,t} \right)$ is  the multiplicity free $\U(2)$-manifold with trivial principal isotropy group whose momentum polytope is the triangle 
\begin{align*}
&r(-\varepsilon_2) + s(\varepsilon_1+\varepsilon_2) + t\cdot \conv(0, \alpha, j\alpha-\varepsilon_2) \\ = \quad
&r(-\varepsilon_2) + s(\varepsilon_1+\varepsilon_2) + t\cdot \conv(0, \varepsilon_1-\varepsilon_2, j\varepsilon_1-(j+1)\varepsilon_2)
\end{align*}  in case (\ref{triangles_delzant}) of \cref{prop:triangles} 
with $\begin{pmatrix}
a_1 & a_2 \\ b_1 & b_2
\end{pmatrix}  = \begin{pmatrix}
j+1 & 1 \\ j & 1
\end{pmatrix}.$  That the invariant compatible complex structures on $M$ from \cref{prop:mfd_triangles_halfright_2} and \cref{prop:mfd_triangles_delzant} are not equivariantly isomorphic can be shown exactly as in the previous case.
\end{remark}

Finally, we describe the multiplicity free $\U(2)$-manifold associated to the momentum polytope (\ref{triangles_right}) of \cref{prop:triangles}.

\begin{proposition} \label{prop:mfd_triangles_right}
Let 
\[
M = \SO(5)/[\SO(2)\times \SO(3)]
\]
be the Grassmannian of oriented $2$-planes in $\R^5$. We give $M$ the structure of a Hamiltonian $\SO(5)$-manifold by viewing it as the coadjoint orbit through the short roots of $\SO(5)$, with respect to the maximal torus $S=\left\{\left(\begin{smallmatrix}A & 0 & 0 \\ 0 & B & 0 \\ 0 & 0 & 1\end{smallmatrix} \right): A,B \in \SO(2)\right\}$. We define an embedding $\iota: \U(2) \into \SO(5)$ by embedding $\SO(4)$ into $\SO(5)$ as the upper left block and identifying $\U(2)$ with the centralizer of $\left\{\left(\begin{smallmatrix}A & 0 \\ 0 & A \end{smallmatrix} \right): A \in \SO(2)\right\}$ in $\SO(4)$ in such a way that the restriction of $\iota$ to $T$ is an isomorphism from $T$ onto $S$ that identifies the shorts roots of $\SO(5)$ with the four weights $-\varepsilon_1, -\varepsilon_2,\varepsilon_1, \varepsilon_2 \in \wl$ of $\U(2)$. 

Let $\mu_M: M \to \fu(2)^*$ be the momentum map and $\om_M$ be the symplectic form of the restricted Hamiltonian $\U(2)$-action on $M$ induced by the inclusion $\iota: \U(2)\into \SO(5)$. 
\begin{enumerate}[(a)]
\item \label{prop:mfd_triangles_right_itema} $(M, \om_M, \mu_M)$ is a multiplicity free $\U(2)$-manifold with trivial principal isotropy group whose momentum polytope is the  triangle
\(
\conv(0,\varepsilon_1,-\varepsilon_2),
\)
in case (\ref{triangles_right}) of \cref{prop:triangles}.
\item \label{prop:mfd_triangles_right_itemb} If $s \in \R$ and $t \in \Rp$, then $(M, t\om_M, t\mu_M+s(\varepsilon_1+\varepsilon_2))$ is a multiplicity free $\U(2)$-manifold with trivial principal isotropy group whose momentum polytope is the  triangle
\[s(\varepsilon_1+\varepsilon_2)+ t\cdot\conv(0,\varepsilon_1,-\varepsilon_2),
\]
of case (\ref{triangles_right}) of \cref{prop:triangles}.
\end{enumerate}
\end{proposition}
\begin{proof}
Since (\ref{prop:mfd_triangles_right_itemb}) follows from  
(\ref{prop:mfd_triangles_right_itema}) and \cref{lem:stretch_and_shift}, we only need to prove part (\ref{prop:mfd_triangles_right_itema}).
Let  $r: \fu(2)^* \to \ft^*$ be the restriction map.  Then $\mu_T = r \circ \mu_M$ is the momentum map of the restricted $T$-action on $M$. The momentum polytope $\mop_T(M)=\mu_T(M)$ of this restricted $T$-action was computed in 
\cite[Example 4.2]{cho_kim_hamiltonian_arxiv}
to be the following square:
\begin{equation} \label{eq:T_right_triangle}
\begin{tikzpicture}
\draw[step=2, dotted, gray] (-3,-3) grid (3,3);
\draw[very thick] (0,2) -- (2,0) -- (0,-2) -- (-2,0) -- (0,2) --(0,-2) -- (2,0) -- (-2,0);
\draw[very thick, dashed] (-3,-3)--(3,3);
  \node at (0,2)[circle,fill,inner sep=2pt]{};
  \node at (2,0)[circle,fill,inner sep=2pt]{};
  \node at (-2,0)[circle,fill,inner sep=2pt]{};
  \node at (0,-2)[circle,fill,inner sep=2pt]{};
  %\node at (1.4,1){$\gamma$};
  \node at (2.35,0){$\varepsilon_1$};
   \node at (-2.45,0){$-\varepsilon_1$};
   \node at (0,-2.30){$-\varepsilon_2$};
    \node at (0,2.30){$\varepsilon_2$};
\end{tikzpicture}
\end{equation}
In this picture, the lines (also the ones in the interior of the momentum image) are the images under $\mu_T$ of the points of $M$ with nontrivial $T$-isotropy, and the dots are the images of the four $T$-fixed points. Our goal is to show that the $\U(2)$-momentum polytope 
$\mop(M)$ of $M$ is 
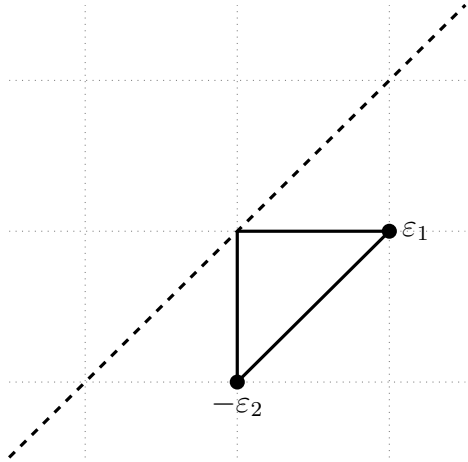
\begin{center}
\begin{tikzpicture}
\draw[step=2, dotted, gray] (-3,-3) grid (3,3);
\draw[very thick] (0,0) -- (2,0) -- (0,-2) -- (0,0);
\draw[very thick, dashed] (-3,-3)--(3,3);
  \node at (2,0)[circle,fill,inner sep=2pt]{};
  \node at (0,-2)[circle,fill,inner sep=2pt]{};
  \node at (2.35,0){$\varepsilon_1$};
    \node at (0,-2.30){$-\varepsilon_2$};
\end{tikzpicture}
\end{center}
and that $M$ is a multiplicity free $\U(2)$-Hamiltonian manifold with trivial principal isotropy group. 

By  \cref{thm:sjamaar_mp}(\ref{item_pol_vertex_interior}), any vertex of $\mop(M)$ that lies in the interior of $\ft_+$ is the image under $\mu_M$ of a $T$-fixed point. Together with \cref{prop:convexhullweylorbit} it follows that $\varepsilon_1$ and $-\varepsilon_2$ are the only two vertices of $\mop(M)$ in $\ft_+$. In order to show that $\mop(M)$ is the asserted triangle, we now only need to prove that the two points where the boundary of the $T$-momentum image $\mop_T(M)$ intersects the Weyl wall do not lie in $\mop(M)$.  Let $q$ be the  $T$-fixed point on $M$ with $\mu_M(q) = \varepsilon_1$. Then the orbit $\U(2)\cdot q \cong \U(2)/T\cong S^2$ is, via the $T$-momentum map $\mu_T$, mapped onto the line segment between $\varepsilon_1$ and $\varepsilon_2$. This implies that the weights of the $T$-representation on the symplectic slice $N_q$ in $q$ are given by the directions of the other two rays  emerging from $\varepsilon_1$ in \eqref{eq:T_right_triangle}. Then \cref{thm:sjamaar_mp}(\ref{item_pol_vertex_cone}) implies that $\mop(M)$  has the desired form locally around $\varepsilon_1$. Together with similar considerations near $-\varepsilon_2$, this forces $\mop(M)$ to be globally as claimed.

Next we show the claim that $M$ contains points with trivial isotropy. A neighborhood of $q$ is $\U(2)$-equivariantly diffeomorphic to $\U(2) \times_T N_q$ (see \cref{rem_thm_sjam_mp}(\ref{rem_thm_sjam_mp_slice})). Note that  $T$ acts  on the symplectic slice $N_q$ with two weights which form a basis of the weight lattice $\wl$, because they are a long and a short root of $\SO(5)$. The claim  follows.  Since $M$ has dimension $6$, \cref{eq:dim_mf} now yields that $M$ is a multiplicity free $U(2)$-manifold.  
\end{proof}

\section{Diffeomorphism types} \label{sec:difftypes}

In this final section we discuss the nonequivariant diffeomorphism types of the manifolds in \cref{table_triangle_manifolds}. We start off with a brief review of some standard facts in the theory of
(real or complex) vector bundles $V\to E\stackrel{\pi}{\to} S^k$ with structure group $G \subset \GL(V)$ over spheres (for details, see \cite[Section 1.2]{Hatcher_VBKT_2_2}). Denote by $N$ and $S$
the north and south pole of the $k$-sphere $S^k$ ($k\geq 2$), respectively. Then both $U^-:=S^k\setminus \{N\}$ and $U^+=S^k\setminus \{S\}$ are homeomorphic to the open $k$-disk $U^k$ and therefore contractible,
so there are trivializations
\[
\phi^- \colon \pi^{-1}(U^-)\to U^- \times V,\quad \phi^+ \colon \pi^{-1}(U^+)\to U^+ \times V.
\]
Now, as $U^- \cap U^+ \cong S^{k-1}\times (-1,1)$, we obtain a map
\[
\phi^+ \circ \phi^- \colon S^{k-1}\times (-1,1)\times V\to S^{k-1}\times (-1,1) \times V,
\]
which is of the form $(x,t,v)\mapsto (x,t,\gamma(x,t)(v))$, for a map $\gamma \colon S^{k-1}\times (-1,1)\to G$.
In particular, the map
\[\gamma_E: S^{k-1} \to G, \gamma_E(x) = \gamma(x,0),\]
 defines an element $[\gamma_E]$ in the set $[S^{k-1},G]$ of free homotopy classes of maps from $S^{k-1}$ to $G$.

Conversely, given an element $[\gamma]\in [S^{k-1},G]$ represented by $\gamma\colon S^{k-1}\to G$, we can define a bundle
$V\to E_{\gamma}\stackrel{\pi}{\to} S^k$ by gluing two copies of $D^k\times V$, where $D^k$ is the closed $k$-disk, together via $\gamma$.
More precisely, writing $D^-\times V$ and $D^+\times V$ for the two copies of $D^k \times V$,  we define
\[
E_{\gamma}:= (D^-\times V) \cup_{\phi_{\gamma}} (D^+\times V),
\]
where $\phi_{\gamma}\colon \partial D^-\times V \to \partial D^+\times V$ is  given by $\phi_{\gamma}(x,v)=(x,\gamma(x)(v))$. This is called
the \textit{clutching construction} and $\gamma$ the \textit{clutching function}. It turns out that
the isomorphism class of this bundle only depends on the free homotopy class of $\gamma$ and that this construction
inverts the assignment $[E] \mapsto [\gamma_E]$ described above. In summary, we have the following
\begin{theorem}\label{bundle}
	The map from $[S^{k-1},G]$ to the set of isomorphism classes of vector bundles $V\to E\to S^k$ with
	structure group $G\subset \GL(V)$, which is given by mapping $[\gamma]\in [S^{k-1},G]$ to the isomorphism class of the bundle $E_{\gamma}$,
	is a bijection. Its inverse is given by the assignment $[E] \mapsto [\gamma_E]$. 
\end{theorem}

Recall that the set $\Vect^1(S^2)$ of isomorphism classes of complex line bundles over $S^2$ is an abelian group with respect to the tensor product operation. \Cref{bundle} gives us the bijection $\Vect^1(S^2) \to [S^1,\GL(1,\C)], [E] \mapsto [\gamma_E]$. Since  $S^1 = \U(1) \subset  \GL(1,\C)$ is a deformation retract of $\GL(1,\C) = \C^{\times}$ we can identify $ [S^1,\GL(1,\C)]$ with $[S^1,S^1] = \pi_1(S^1)$. 
Also, by definition, the tensor product of two line
	bundles $E_1$ and $E_2$ has the clutching function $\gamma_{E_1}\cdot \gamma_{E_2}$ (multiplying in $S^1=\U(1)$), which
	makes the assignment \[\Vect^1(S^2) \to \pi_1(S^1):[E]\mapsto [\gamma_E]\] a group homomorphism and, by \cref{bundle}, a group isomorphism.  

We now fix group isomorphisms $\phi_1: \Vect^1(S^2) \to \Z$ and $\phi_2: \pi_1(S^1) \to \Z$.  Such isomorphisms are unique up to sign, but it will turn out that the choice of sign will not be important in what follows. 
By abuse of notation, we will  write $\phi_1(E)$ for $\phi_1([E])$ and $\phi_2(\gamma)$ for $\phi_2([\gamma])$. Since $H^2(S^2,\Z) \cong \Z$, 
$\phi_1(E)$ can be understood as the Chern class of the complex line bundle $E$ up to sign, see e.g. \cite[Proposition 3.10]{Hatcher_VBKT_2_2}.
\begin{lemma}\label{Chern}
	Let $S^1$ act on $S^2$ by standard rotation and on two copies of $\C$ via weights
	$k_1 \in \Z$ and $k_2\in \Z$, respectively.
	Consider the corresponding $S^1$-equivariant line bundle $\C \to E \to S^2$ with weight $k_1$ on the fiber at the south pole $S$ and $k_2$ on the fiber at the north pole $N$.
	Then $\phi_1(E)=\pm(k_1-k_2)$, depending on the chosen $\phi_1$. 
\end{lemma}
\begin{proof}
	We only have to determine $\phi_2(\gamma)$ up to sign, where $\gamma: S^1 \to S^1$ is the clutching function of the line bundle $E$.
	Trivializations of $E$ around $S$ and $N$ look like $D^2\times \C$ with $S^1$-actions
	\begin{align*} 
		s\cdot (z_1,z_2)=(sz_1, s^{k_1} z_2) \quad \text{ and } \quad s\cdot (z_1,z_2)=(sz_1, s^{k_2} z_2),
	\end{align*}
	respectively. The isomorphism between the boundaries of these two trivializations induced by the clutching function
	$\gamma$ has to preserve this $S^1$-action, which gives the condition (now $z_1\in \partial D^2=S^1$)
	\begin{align*}
		(s z_1, \gamma(s z_1) s^{k_2} z_2)=(s z_1, s^{k_1} \gamma(z_1) z_2).
	\end{align*}
	This immediately implies that $\phi_2(\gamma)$ is $\pm(k_1-k_2)$, the sign depending on the choice of $\phi_2$.
\end{proof}

\begin{theorem} \label{thm:difftypes}
	There are precisely four diffeomorphism types occurring in \cref{table_triangle_manifolds}:
	\begin{enumerate}[(a)]
		\item the manifolds in case (\ref{triangles_side_wall}) are diffeomorphic to $\PP(\C^4)$,
		\item the manifold in case (\ref{triangles_right}) has the diffeomorphism type of the Grassmannian of oriented $2$-planes in $\R^5$,
		\item those manifolds $\hfb{\U(2)}{T}{\PP(V)}$ in cases (\ref{triangles_delzant}), (\ref{triangles_halfright_1}) and (\ref{triangles_halfright_2}) for which the  first Chern class of the vector bundle $V\to \hfb{\U(2)}{T}{V}\to \U(2)/T$ is divisible by $3$
		are diffeomorphic to
		$S^2\times \PP(\C^3)$,
		\item 
 those manifolds $\hfb{\U(2)}{T}{\PP(V)}$ in cases (\ref{triangles_delzant}), (\ref{triangles_halfright_1}) and (\ref{triangles_halfright_2}) for which the  first Chern class of the vector bundle $V\to \hfb{\U(2)}{T}{V}\to \U(2)/T$ is not divisible by $3$		
		are diffeomorphic to
		the total space of any non-trivial $\PP(\C^3)$-bundle over $S^2$.
	\end{enumerate}
\end{theorem}
\begin{proof}
	As $\PP(\C^4)$ is spin and the aforementioned Grassmannian is not, these two manifolds are not diffeomorphic. In addition, both of them are not diffeomorphic
	to the manifolds occurring in cases (\ref{triangles_delzant}), (\ref{triangles_halfright_1}) and (\ref{triangles_halfright_2}) of \cref{table_triangle_manifolds} due to the equality of Euler characteristics
	$\chi(M)=\chi(M^T)$ which holds for any torus action on a compact manifold $M$, see \cite{kobayashi-fixed_points_isometries}. The real task here is to distinguish between the manifolds in
	cases (\ref{triangles_delzant}), (\ref{triangles_halfright_1}) and (\ref{triangles_halfright_2}).
	
		Let $M = \hfb{\U(2)}{T}{\PP(V)}$ be one of these manifolds. As the projective bundle of the vector bundle $E=\hfb{\U(2)}{T}{V}$ of  rank $3$ over $\U(2)/T \cong S^2$, it can be  described by a clutching function $\gamma \colon S^1 \to \text{PGL}(3,\C)$, which comes from the clutching 
	function $\tilde{\gamma}$ of $E$. Because $E$ is the sum $L_1 \oplus L_2 \oplus L_3$ of three line bundles, we have $\tilde{\gamma}=(\gamma_1,\gamma_2,\gamma_3)\colon S^1\to \U(1)^3 \subset \U(3)$, where $\gamma_1, \gamma_2$ and $\gamma_3$ are the clutching functions of $L_1, L_2$ and $L_3$.
	The class of $\tilde{\gamma}$ in $\pi_1(\U(3))=\pi_1(\text{GL}(3,\C))=\Z$ is now given by
	$$\phi_2(\det(\tilde{\gamma}))=\phi_2(\gamma_1)+\phi_2(\gamma_2)+\phi_2(\gamma_3).$$
	It follows that the class of $\gamma$ in $[S_1,\PGL(3,\C)]=\pi_1(\text{PGL}(3,\C))=\Z/3\Z$ is determined by the value
	\[
	f(\gamma):=[\phi_2(\gamma_1\cdot \gamma_2 \cdot \gamma_3)]=[\phi_2(\gamma_1)+\phi_2(\gamma_2)+\phi_2(\gamma_3)] \in \Z/ 3\Z,
	\]
	since the fibration $Z(\text{GL}(3,\C))\to \text{GL}(3,\C)\to \text{PGL}(3,\C)$ induces a short exact sequence
	\[
	0\to \pi_1(Z(\text{GL}(3,\C)))\to \pi_1(\text{GL}(3,\C))\to \pi_1(\text{PGL}(3,\C))\to 0.
	\]
	Note that $f(\gamma)$ is equal up to sign to $\phi_1(L_1)+\phi_1(L_2)+\phi_1(L_3)$ modulo $3$, where the $L_i$ are the line bundles from above
	and $\phi_1$ is the fixed isomorphism $\Vect^1(S^2) \to \Z$.
	
	We only need to check that the total spaces $E_1$ and $E_{-1}$ of the $\PP(\C^3)$-bundles with $f(\gamma_+)=1$ and $f(\gamma_-)=-1$,
	where $\gamma_+: S^1 \to \PGL(3,\C)$ and $\gamma_-: S^1 \to \PGL(3,\C)$ are the clutching functions of $E_1$ and $E_{-1}$, are diffeomorphic, and that
	$E_1$ and $S^2\times \PP(\C^3)$ are not (note that all these statements do not depend on the isomorphisms $\phi_1$ and $\phi_2$ we have chosen).

	Because the vector bundles of $E_{\pm 1}$ are sums of three line bundles, the first statement follows immediately from the fact that two complex line bundles over $S^2$, whose first Chern classes differ only in their sign, are $\C$-antilinearly isomorphic (as a change in sign of the first Chern class corresponds to a change in sign of the
	complex structure on the fiber). The second statement is true as $E_1$ and $S^2\times \PP(\C^3)$ are not even homotopy
	equivalent. Indeed, by e.g. \cite[\S 17.2]{husemoller-fb_gtm_3ed}, the cohomology ring of $E_1$ is
	$\Z[x,y]/(x^2, y^3+xy^2)$, where $x$ represents a generator of $H^*(S^2)$ and $y$ represents a generator of $H^*(\PP(\C^3))$, whereas $H^*(S^2\times \PP(\C^3)) = \Z[x,y]/(x^2,y^3)$. These two cohomology rings are not isomorphic since
	any graded ring isomorphism
	\begin{align*}
		\Z[x,y]/(x^2, y^3) \to \Z[x,y]/(x^2, y^3+xy^2)
	\end{align*}
	would have to send $x$ to $\pm x$ and therefore $y$ to $ax\pm y$ for some $a \in \Z$,  but $y^3=0$ on the left, whereas
	$(ax\pm y)^3=3axy^2\pm y^3 \neq 0$ on the right.
\end{proof}
\begin{remark} 
	In order to determine the first Chern class modulo $3$ of the $\C^3$-bundle $E$ giving the $\PP(\C^3)$-bundle $M$ of case (\ref{triangles_delzant}), (\ref{triangles_halfright_1}) or (\ref{triangles_halfright_2}) in \cref{table_triangle_manifolds}, it is sufficient to look
	at the directions $\lambda_1=a_1 \eps_1+b_1 \eps_2$ and $\lambda_2=a_2 \eps_1+b_2 \eps_2$ in which the edges of
	the  momentum polytope $\mop(M)$ emerge at some vertex $\sv$. 
	Indeed, a neighborhood of the $\U(2)$-orbit  $\Psi^{-1}(\sv)$ in $M$ looks like the bundle
	$L'=\U(2)\times_{T} (\C_{-\lambda_1} \oplus \C_{-\lambda_2})$.
	Now consider the action of $\U(1)\times \{e\}\subset T \subset \U(2)$ on $L'$ and note that the weights of that circle action on the fiber over
	$eT \in \U(2)/T$ are given by $-a_1$ and $-a_2$, whereas the weights in the fiber over the other $T$-fixed point in $\U(2)/T$ are $-b_1$ and $-b_2$.
	Using \cref{Chern}, we see that the first Chern class of $L'$ is (up to sign) equal to  $-a_1-a_2+b_1+b_2$.
	Now observe that $\PP(L'\oplus \underline{\C})=\PP(E) = M$, which implies that $M$ is diffeomorphic
	to $S^2\times \PP(\C^3)$ if and only if $a_1+a_2-b_1-b_2$ is a multiple of $3$.
\end{remark}

%%bibliography
\def\cprime{$'$} \def\cprime{$'$} \def\cprime{$'$} \def\cprime{$'$}
  \def\cprime{$'$}
\providecommand{\bysame}{\leavevmode\hbox to3em{\hrulefill}\thinspace}
\providecommand{\MR}{\relax\ifhmode\unskip\space\fi MR }
\providecommand{\MRhref}[2]{%
  \href{http://www.ams.org/mathscinet-getitem?mr=#1}{#2}
}
\providecommand{\href}[2]{#2}

\end{document}